%% file: AccessibilityRevised3.tex
\documentclass[11pt, twoside]{article}
\usepackage{amsfonts, amssymb,amsmath,amsthm}
\usepackage{pstricks,psfrag,xmpmulti,amscd}
\usepackage{graphicx}
 \usepackage{import}
\usepackage[all]{xy}
 \divide\oddsidemargin by 2
\input{Tesimacros.tex}

\newcommand{\eucl}{{\text{eucl}}}
\newcommand{\leucl}{\ell_{\text{eucl}}}
\newcommand{\lhyp}{\ell_{\Omega}}

\newcommand{\szero}{{{\mathbf{s}}_0}}
\newcommand{\sn}{{{\mathbf{s}}_n}}

\renewcommand{\a}{{\mathbf{a}}}
\newcommand{\an}{{{\a_n}}}
\newcommand{\s}{{\mathbf{s}}}

\newcommand{\stilde}{{\widetilde{\s}}}
\renewcommand{\H}{{\mathbb{H}}}
\newcommand{\gs}{g_{\s}}
\newcommand{\gn}{{g_{\sn}}}
\newcommand{\gst}{g_{\tilde{\s}}}
\newcommand{\gz}{{g_{\szero}}}

\newcommand{\Boettcher}{B\"ottcher }

\newcommand{\rhoeucl}{{\rho_{\text{eucl}}}}
\newcommand{\rhohyp}{{\rho_{\Omega}}}
\newcommand{\sigmam}{{\sigma^{-1}}}

\title{Repelling periodic points and landing of rays for post-singularly
bounded exponential maps}
\author{A. M. Benini, M. Lyubich}

\begin{document}

\maketitle
\begin{abstract}
\emph{We  show that repelling periodic points are landing points of periodic rays for exponential maps whose singular value has bounded orbit. For polynomials with connected Julia sets, this is a celebrated theorem by Douady, for which we  present a new proof. In both cases we also show that points in hyperbolic sets are accessible by at least one and at most finitely many rays. For exponentials this allows us to conclude that the singular value itself is accessible.}
\end{abstract}

\section{Introduction}

Let $f:\C\ra\C$ be an entire function. Then the dynamical plane $\C$ splits into two completely invariant subsets, the Fatou set $\FF(f)$, on which the dynamics is stable, and its complement, the Julia set $J(f)$, on which the dynamics is chaotic.
More precisely, the Fatou set is defined as

\[
\FF(f):=\left\{ z\in\C: \{f^n\}\text{ is normal in a neighborhood of $z$}\right\}.
\] 

\noindent In this paper, we will consider the case in which $f$ is either a polynomial or a complex exponential map $e^z+c$. An important role is played by the  {\em escaping set} 
\[
I(f):=\{z\in\C: \ f^n(z)\ra \infty \text{\ as $n\to \infty$}\}.
\]
\noindent
For polynomials, $I(f)\subset \FF(f)$, while for exponentials $I(f)\subset J(f)$ (see \cite{BR}, \cite{EL}).
However, in both cases the escaping set can be described as an uncountable collection of injective curves, called {\em dynamic rays} or just {\em rays}, which tend to infinity on one side and are equipped with some symbolic dynamics (see Sections~\ref{Accessibility of repelling periodic orbits for polynomials with connected Julia set} and \ref{Dynamic rays in the exponential family}). While in the polynomial case either a point is escaping or its orbit is bounded, in the exponential case  it might well be that orbits form an unbounded set without tending to infinity.

For a polynomial of degree $D$ with connected Julia set, $I(f)$ is an open topological disk centered at infinity, and the dynamics of $f$ on $I(f)$ is conjugate to the dynamics of $z^D$ on $\C\setminus\ov{\D}$ via the B\"ottcher map {(see Section~\ref{Accessibility of repelling periodic orbits for polynomials with connected Julia set})}. In this case, dynamic rays are defined simply as preimages of straight rays under the B\"ottcher map, and the symbolic dynamics on them is inherited from the symbolic dynamics of $z^D$ on the unit circle $\R/\Z$ (see e.g. \cite{Mi}).
For exponentials, we refer the reader to \cite{SZ1} and to Section~\ref{Dynamic rays in the exponential family} of this paper.
A ray  is called {\em periodic} if it is mapped to itself under some iterate of the function.

It is important to understand the interplay between the rays and the set of non-escaping points. 
A ray $g_\s$ is said to {\em land} at a point $z$ if $\ov{g_\s}\setminus g_\s =\{z\}$; conversely, a point is {\em accessible} if it is the landing point of at least one ray. 

 Ideally every ray lands and  every non-escaping point in the Julia set is accessible,  like for hyperbolic maps (in both polynomial and exponential setting).
 One weaker, but very relevant, question to ask is whether all periodic rays land  and whether all repelling/parabolic periodic points are accessible by periodic rays. By the Snail Lemma (see e.g. \cite{Mi}) if a periodic ray lands it has to land at a repelling or parabolic periodic point. 
 
Periodic rays are known to land in both the polynomial and the exponential case (see Theorem 18.10 in \cite{Mi} and  \cite{Re1}), unless one of their forward images contains the singular value.

The question whether repelling periodic points are accessible is  harder and still open in the exponential case; in this paper, we give a  positive answer to this problem for an exponential map $f(z)=e^z+c$ whose  \emph{postsingular set}
 $$\PP(f):={\ov{\underset{n>0}\bigcup f^n(c)}}$$ 
 is bounded.
Observe that in this case the singular value is \emph{non-recurrent}, i.e. $c\notin \PP(f)$.

\begin{thmA}
\label{Theorem A}
Let $f$ be either a polynomial or an exponential map, with bounded postsingular set; then any repelling periodic point is the landing point of at least one and at most finitely many dynamic rays, all of which are periodic of the same period.
\end{thmA}

For polynomials with connected Julia set, all repelling periodic points are known to be accessible by a theorem due to Douady (see \cite{H}). Another proof due to Eremenko and Levin can be found in \cite{ELv}:  their proof covers also the case in which the Julia set is disconnected. However, neither proof can be generalized to the exponential family, because both use in an essential way the fact that the basin of infinity is an open set.
Our proof of  Theorem A also gives a new proof in the polynomial setting (see Section~\ref{Accessibility of repelling periodic orbits for polynomials with connected Julia set}).

Our second result is about accessibility of hyperbolic sets.
A forward invariant compact set  $\Lambda$ is called \emph{hyperbolic} (with respect to the Euclidean metric) if there exist $k\in\N$ and  $\eta>1$ such that  $|(f^k)'(x)|>\eta$ for all $x\in\Lambda$, $k>\ov{k}$. {Observe that hyperbolic sets are subsets of the Julia set.}

\begin{thmB}
\label{Theorem B}
Let $f$ be either a polynomial or an exponential map, with bounded postsingular set. Then any point that belongs to  a hyperbolic set is accessible.
\end{thmB}

Theorem B for polynomials is a special case of  Theorem C in \cite{Pr}.
Under an additional combinatorial assumption (only needed for polynomials, {since having bounded postsingular set is a stronger assumption in the exponential case than in the polynomial case}) we also show that there are only finitely many rays landing at each point belonging to a hyperbolic set (see Propositions~\ref{Finitely many rays} and~\ref{Finiteness exp}). This is a new result {not only in the exponential case but} also in the polynomial case.

For an exponential map whose postsingular set is bounded and contained in the Julia set,  the postsingular set itself is hyperbolic (see \cite{RvS}, Theorem 1.2). We obtain hence the following  corollary of Theorem B:

\begin{thmC}
Let $f$ be an exponential map with bounded postsingular set. Then every point in the postsingular set, and in particular the singular value itself, are accessible.
\end{thmC}

Part of the importance of Theorem A is that it gives indirect insight on the structure of the parameter plane. For example, for unicritical polynomials, it implies that there are no irrational subwakes attached to hyperbolic components (see Section I.4 in \cite{H}, Theorem 4.1 in \cite{S1}).
The proof  in \cite{S1} is combinatorial and  can be applied to the exponential family (see Section~\ref{A remark about parabolic wakes}). The results in this paper are also used in \cite{Be} to show rigidity for non-parabolic exponential parameters with bounded postsingular set.

The structure of this article is as follows:
in Section~\ref{Accessibility of repelling periodic orbits for polynomials with connected Julia set} we introduce some background about polynomial dynamics; we then  present the  new proof of Douady's theorem about accessibility of repelling periodic points, followed by the proof of Theorem B in the polynomial case. The proof in this paper only uses quite weak information about the structure of dynamic rays, opening up this result to be generalized to other families of functions beyond the exponential family.
In Section~\ref{Dynamic rays in the exponential family} we recollect some facts on exponential dynamics, including existence and properties of dynamic rays in this case. In Section~\ref{Accessibility for exponential parameters with bounded postsingular set} 
we state and prove Theorems A and B in the exponential setting.
More precisely, in  Section~\ref{Bounds on fundamental domains for exponentials} we make some estimates about the geometry of rays near infinity that are needed to prove Theorems A and B for exponentials; the proofs themselves are presented in Section~\ref{Proof of accessibility theorems}.

We  denote by $\leucl(\gamma)$ the Euclidean length of a curve $\gamma$ and by $\ell_\Omega(\gamma)$ its hyperbolic length in a region $\Omega$ admitting the hyperbolic metric with density $\rho_\Omega$. A ball of radius $r$ centered at a point $z$ is  denoted by either $B_r(z)$ or $B(z,r)$.

\subsection*{Acknowledgments}
The first author would like to thank  Carsten Petersen for interesting discussions, as well as  the IMS  in Stony Brook and the IMATE in Cuernavaca for hospitality. We would also like to thank the referee for suggesting several improvements and clarifications. 
This work was partially supported by NSF. The first author was partially supported by the ERC grant HEVO - Holomorphic Evolution Equations n. 277691. 

\section{Accessibility of repelling periodic orbits for polynomials with connected Julia set}
\label{Accessibility of repelling periodic orbits for polynomials with connected Julia set}

In this section we  give a new proof of Douady's   theorem  for polynomials with connected Julia set, showing that  any repelling periodic orbit is the landing point of finitely many periodic rays. 

Let $f$ be a  polynomial of degree $D$ with connected Julia set $J(f)$ and filled Julia set  $K(f)$.  As $K(f)$ is full and contains more than one point, $\Omega=\C\setminus   K(f)$ is a domain that  admits a hyperbolic metric with some density $\rho_\Omega(z)$. Since $f:\C\setminus K(f) \ra \C\setminus K(f) $ is a covering map, it locally preserves the hyperbolic metric.

Since $K(f)$ is connected there exists a unique conformal isomorphism ({(see \cite[Chapter 18]{Mi} and \cite[Theorem 9.5]{Mi}}), called the  B\"ottcher function $B$, which  conjugates the dynamics of $f$ on $\C\setminus   K(f)$ to the dynamics of $z^D$ on $\C\setminus  \ov{\D}$ and which is asymptotic to the identity map at infinity. 

Points in $\C\setminus \ov{\D}$ can be written as $z=e^t e^{2\pi i \s}$ with $t\in (0,\infty)$ and $\s\in \R/\Z$. We refer to $t$ as the \emph{potential} of $z$ and to $\s$ as the \emph{angle} of $z$. Consider the straight ray of angle $\s$ parametrized as  $R_\s(t)=e^t e^{2\pi i \s}$, with fixed $\s\in \R/\Z$  and $t\in(0,\infty)$; we define the curve  $g_\s(t):=B^{-1}(R_\s(t))$   to be  the \emph{dynamic ray} of angle $\s$, and we keep referring to $t$ as the potential.

The space $\Sigma_D$ of infinite  sequences over $D$ symbols projects naturally to  $\R/\Z$ via the $D$-adic expansion of numbers in $\R/\Z$, and this projection is one-to-one except for countably many points corresponding to $D$-adic numbers. 
The same  projection semiconjugates multiplication by $D$ on angles in $\R/\Z$  to the  left-sided  shift map $\sigma$ acting on $\Sigma_D$.
In this paper we will use both models. The sequences in $\Sigma_D$ in our paper represent angles, so we will often use the term 'angle'  when referring to them.

For $\s=s_0 s_1 s_2\ldots$ and $\s'=s'_0 s'_1 s'_2\ldots$ in $\Sigma_D$ we consider the distance
\begin{equation}\label{D metric}
 |\s-\s'|_D=\underset{i}\sum \frac{|s_i-s_i'|}{D^{i+1}}.
\end{equation}

We will denote by $|\s-\s'|$ the usual distance between  angles in $\R/\Z$. 
A continuous  map $\phi$ acting on a compact metric space $(X,d)$ is called \emph{locally expanding} if there exists $\rho>1$ and $\epsilon>0$ such that
$
  d(\phi( x), \phi(y)) \geq \rho\, d(x,y)
$
for any two points $x,y\in X$ with $d(x,y)< \epsilon$. 
 Both  the shift map $\sigma$ on $\Sigma_D $ and multiplication by $D$ on $\R/\Z$ are locally expanding by a factor $D$. 

Define the radial growth function as $F: t\mapsto D t$. Then since  $R_\s(t)\mapsto R_{\sigma\s}(F(t))$ under the map $z\mapsto z^D$ we also have 
\begin{equation}\label{Functional} f(g_{\s}(t))=g_{\sigma \s}(F(t)).\end{equation}

A \emph{fundamental domain} starting at $t$ for a ray $g_{\s}$ is the arc $g_{\s}([t, F(t)))$, and is  denoted by $I_t(g_{\s})$.

The following lemma relates convergence of angles to convergence of dynamic rays.

 \begin{lem}\label{Angle-space continuity}
 Let $f$ be a polynomial of degree $D$.
 For each $t_*, t^*>0$, the rays $g_{\sn}(t)$ converge uniformly to the ray $g_\s(t)$ on $[t_*,t^*]$ as $\sn \ra \s$. 
 \end{lem} 
 \begin{proof}
{The inverse of the B\"ottcher map is uniformly continuous on the closed annulus $\{z\in \C: z=e^te^{2\pi i \s} \text{ with } t\in [t_*,t^*],\ {\s\in \R/\Z}\}$. 
 Since the straight rays $R_\sn(t)$ converge uniformly to the ray $R_\s(t)$ on the compact set $[t_*,t^*]$, the claim follows.}
 \end{proof}
 
 The next lemma gives a sufficient condition to determine when the limit dynamic ray  lands. The proof is the same in  both the polynomial and in the exponential cases, once dynamic rays for the latter have been defined {(see Section~\ref{Dynamic rays in the exponential family} for dynamic rays in the exponential family)}.
 
  \begin{lem}\label{landing}
 Let $x_0\in\C$, $t_0>0,$ and for all $m\in\N$ define $t_m:=F^{-m}(t_0)$. Also let $\gn$ be a sequence of dynamic rays such that $\sn\ra \s$, and such that
 $I_{t_m}(g_\sn)\subset B(x_0,{\frac{A}{\nu^m}})$ for some $A>0,\nu>1$ and for all $n>N_m$.
  Then $g_{\s}$ lands at $x_0$.
 \end{lem}
 \begin{proof}

To show that $g_{\s}$ lands at $x_0$ it is enough to show that for each $m$, 
 
  \[I_{t_m}(g_{\s})\subset  B\left(x_0,{\frac{A}{\nu^m}}\right).\]
  
\noindent For any fixed  $m>0$, $g_\sn\ra g_\s$ uniformly on $[t_m,t_{m-1}]$ by Lemma~\ref{Angle-space continuity}.
As $I_{t_m}(g_\sn)$ is eventually contained in $B\left(x_0,{\frac{A}{\nu^m}}\right)$, taking the limit for $n\ra\infty$ gives the claim.
 \end{proof}

The  following lemma is a consequence of the fact that for points tending to the boundary of a hyperbolic domain, the hyperbolic density tends to infinity.

\begin{lem}\label{Boundary behavior}
Let $\Omega\subset\C$ be a hyperbolic region.
Let $\gamma_n:[0,1]\ra\Omega$ be a family of curves with uniformly bounded hyperbolic length and such that $\gamma_n(0)\ra \partial \Omega$. Then $\leucl(\gamma_n)\ra 0$.
\end{lem}
For reader's convenience, we will outline a proof. 
\begin{proof}
 Passing to a subsequence, we let  $\gamma_n(0)$ converge to some point $z_0\in \partial \Omega$. 
Let us take two more boundary points, $z_1, z_2\in \partial \Omega$, such that 
$\dist (z_i, z_j) \geq \de>0$, $0\leq i < j\leq 2$, where $\operatorname{dist}$ is the spherical distance 
and $\de=\de(\Omega)$. 
By the Schwarz Lemma, the hyperbolic metric on $\Omega$ is dominated by the hyperbolic metric on
$\hat \C\setminus \{z_0,z_1, z_2\}$, so it is enough to prove the statement for the latter domain.   

Let us move the points $(z_0, z_1, z_2) $ to  $(0,  1, \infty) $ by a M\"obius transformation $A$. 
Since the triple $\{z_0, z_1, z_2\}$ is $\de$-separated,  the spherical derivative 
$\| DA(z) \| $ is bounded from below by some constant $c(\de)>0$. Hence it is enough to prove the statement for $\Om= \C\setminus \{0, 1\}$.
 In this case, the universal covering ${\H} \ra \Omega$ is the classical triangle modular function 
$\lambda$ with fundamental domain 
$$
      \{ z\in \H : \ 0\leq  \Re z \leq  2, |z\pm 1/2|\geq 1/2 \} . 
$$

Near $\infty$, this function admits a Fourier expansion of  form $\phi (e^{\pi i z})$, 
where $\phi$ is a univalent function near $0$. Pushing the hyperbolic metric on $\H$ down,
we conclude that the hyperbolic metric on $\Om$ near $0$ is comparable with 
                  $$ \rho(z) |dz | =   \frac {|dz|} {|z|\; | \log |z||} . $$   
Since $\int_0^\eps  \rho (x) dx = \infty$ for any $\eps>0$ (i.e., the cusp has infinite hyperbolic diameter), 
the curves $\gamma_n$ uniformly converge to $0$. Hence 
$  \inf_{\gamma_n} \rho(z) \to \infty\quad as\ n\to  \infty,$
and 
$$
   l(\gamma_n) = \int_{\gamma_n} \frac {\|dz \| }{\rho(z) } \to 0 \quad as\ n\to  \infty,
$$ 
where $\|\cdot \|$ stands for the infinitesimal length of the hyperbolic metric. 
\end{proof}

\subsection{Proof of Theorem  A in the polynomial case}

In this section we give a proof of Theorem A in the polynomial case. 
Up to taking an iterate of $f$, we can assume that the repelling periodic point in question is a repelling fixed point $\alpha$.  Let $\mu>1$ be the modulus of its multiplier, and let $L$ be a linearizing neighborhood  for $\alpha$.
For the branch $\psi$ of $f^{-1}$ fixing $\alpha$, it is easy to show using linearizing coordinates that there exists a $C>0$ such that for all $x\in L$
\begin{equation}\label{Contraction}\frac{1}{C\mu^n}<|(\psi^n)'(x)|<\frac{C}{\mu^n}. \end{equation}

Before proceeding to the proof of Theorem A let us observe the following:

 \begin{prop}\label{Fundamental domains shrinking}
 The Euclidean length $\leucl(I_t(g_{\s}))$ 
 tends to $0$ uniformly in $\s$ as $t\ra 0$.
 \end{prop}
 \begin{proof}
Let $t_*>0$, $t^{*}>F^{2}(t_*)$,  $\Omega=\C\setminus K(f)$. By  uniform continuity of the inverse of the \Boettcher map on compact sets, the Euclidean  length  $\leucl(g_{\s}(t_*,t^*))$, and hence the hyperbolic length $\ell_\Omega(g_{\s}(t_*,t^*))$, are uniformly  bounded in $\s$. 
So by the Schwarz Lemma, $\ell_\Omega(f^{-n}(g_{\s}(t_*,t^*)))$ is also uniformly bounded in $\s$ and $n$, and hence $\ell_\Omega(I_t(g_{\s}))$ is uniformly bounded in $\s$ for $t\leq t_*$ (because for any such $t$, $I_t(g_{\s})\subset f^{-n}(g_{\stilde}(t_*,t^*))$ for some $\stilde\in\R/\Z$ and some $n\geq0$). Since the inverse of the \Boettcher map   is proper,  $\dist(g_{\s}(t),J(f))\ra0$ uniformly in $\s$ as $t\ra0$, hence by Lemma~\ref{Boundary behavior}  $\leucl (I_t(g_{\s}))\ra0$ uniformly in $s$ as $t\ra0$.
\end{proof}

  \begin{thm}\label{Landing of a ray}
 Let $f$ be a polynomial with  connected Julia set, and let $\alpha$ be a repelling fixed point for $f$. Then there is at least one dynamic ray $g_{\s}$ landing at $\alpha.$  
 \end{thm}    

 \begin{proof}
 Let $L$ be a linearizing neighborhood for $\alpha$ and $\mu$ be as in (\ref{Contraction}). 
 Let $U'\subset L$ be a neighborhood of $\alpha$, and let $U$ be its preimage under the inverse branch $\psi$ of $f$ which fixes $\alpha$.  Let 
$\epsilon$ be the distance between $\partial U$ and $\partial U'$.  By Proposition~\ref{Fundamental domains shrinking}, there exists $t_\epsilon$ such that

 \begin{equation}\label{shrinking}\leucl(I_{t}(g_{\s}))<\epsilon \text{ for all }\s\in\Sigma_D, t<t_\epsilon. \end{equation}

 As $\alpha$ is in the Julia set, it is approximated by escaping points with arbitrary small potential $t$, hence there exists a dynamic ray $\gz$ such that $\gz(t_0)$ belongs to $U$ for some $t_0<t_\epsilon$. By (\ref{shrinking}), $\leucl(I_{t_0}(\gz))\leq \epsilon$, hence $I_{t_0}(\gz)\subset U'$ (See Figure~\ref{A}).

\begin{figure}[hbt!]
\begin{center}
\def\svgwidth{5cm}

\begingroup
  \makeatletter
  \providecommand\color[2][]{%
    \renewcommand\color[2][]{}%
  }
  \providecommand\transparent[1]{
   
    \renewcommand\transparent[1]{}%
  }
  \providecommand\rotatebox[2]{#2}
  \ifx\svgwidth\undefined
    \setlength{\unitlength}{466.06061386pt}
  \else
    \setlength{\unitlength}{\svgwidth}
  \fi
  \global\let\svgwidth\undefined
  \makeatother
  \begin{picture}(1,0.86990702)%
    \put(0,0){\includegraphics[width=\unitlength]{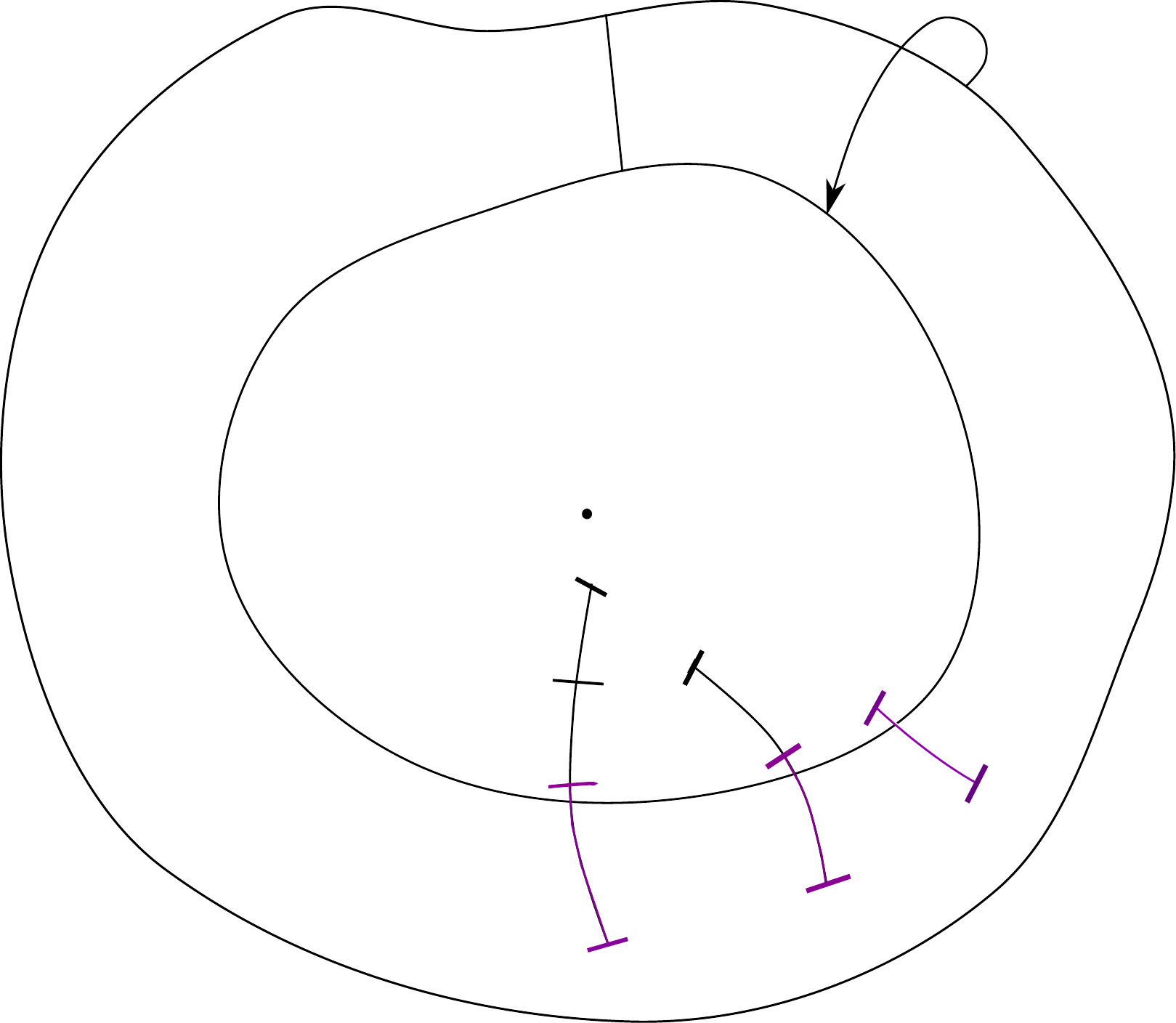}}%
    \put(0.53367775,0.78105621){\color[rgb]{0,0,0}\makebox(0,0)[lb]{\smash{$\epsilon$}}}%
    \put(0.5169038,0.44557696){\color[rgb]{0,0,0}\makebox(0,0)[lb]{\smash{$\alpha$}}}%
    \put(0.60077364,0.27364386){\color[rgb]{0,0,0}\makebox(0,0)[lb]{\smash{$\gamma_1$}}}%
    \put(0.51271026,0.34353537){\color[rgb]{0,0,0}\makebox(0,0)[lb]{\smash{$\gamma_2$}}}%
    \put(0.77410445,0.22751545){\color[rgb]{0,0,0}\makebox(0,0)[lb]{\smash{$\gamma_0$}}}%
    \put(0.80765245,0.54342507){\color[rgb]{0,0,0}\makebox(0,0)[lb]{\smash{$U$}}}%
    \put(0.92227456,0.68041242){\color[rgb]{0,0,0}\makebox(0,0)[lb]{\smash{$U'$}}}%
    \put(0.74614791,0.75030394){\color[rgb]{0,0,0}\makebox(0,0)[lb]{\smash{$\psi$}}}%
  \end{picture}%
\endgroup

\end{center}
\caption{\small Construction of the curves $\gamma_n$ in the proof of Theorem~\ref{Landing of a ray}.}\label{A}
\end{figure}
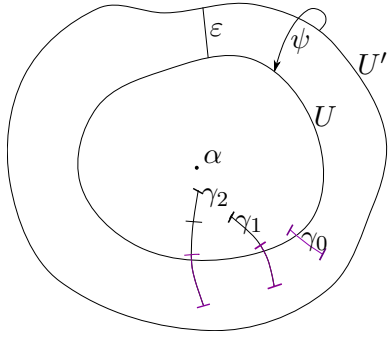

 \noindent For any $n>0$ recall that  $\psi^n$ is the branch of $f^{-n}$ fixing $\alpha$, and let $\gn$ be the ray containing $\psi^n (I_{t_0}(g_\szero))$. Observe that  $\sigma^n \s_n=\szero$ for all $n$.
 Let us define inductively a sequence of curves $\gamma_n\subset \gn$ as 
 \begin{align*}
 \gamma_0 &:=I_{t_0}(\gz)  \\
 \gamma_n &:= \psi (\gamma_{n-1})\cup I_{t_0}(\gn).
 \end{align*}
We show inductively that the curves $\gamma_n$ are well defined and that  the following properties hold: 
 \begin{enumerate}
 \item $\gamma_n= \gn(t_n,F(t_0))$, where $t_n:=F^{-n}(t_0)$;
 \item $\gamma_n\subset U'$ for all $n$;
 \item  $I_{t_m}(\gn)\subset B\left(\alpha,\frac{ C\diam U' }{\mu^m}\right)$, for all $m\leq n.$ \label{estimates}
 \end{enumerate}
 All properties are true for $\gamma_0$, so let us suppose that they hold for $\gamma_{n-1}$ and show that they also hold for $\gamma_n$.
We have that
 \[\psi (\gamma_{n-1})=\psi(g_{\s_{n-1}}(t_{n-1}, F(t_0)))=\gn(t_n, t_0)\] 
 by the functional equation (\ref{Functional}) and  by the definition of $\gn$. Also, $\psi(\gamma_{n-1})\subset U$ because by the inductive assumption $\gamma_{n-1}\subset U'$ and $\psi(U')=U$. As  $\leucl(I_{t_0}(\gn))\leq\epsilon$, and $\gn(t_{n},t_0)\subset U$, we have that $\gamma_n\subset U'$.

 If $x\in I_{t_m}(\gn)$ for $\ m\leq n$, then $x=\psi^m y$ for some $y\in I_{t_0}(g_{\s_{n-m}})\subset U'$, hence by (\ref{Contraction}) we have
\[ |x-\alpha|\leq\frac{C|y-\alpha|}{\mu^m}\leq  \frac{C\diam U' }{\mu^m},     \]
 proving Property~\ref{estimates}.
 
As the sequence  $\{\sn\}$  of angles of the rays $\gn$ is contained in $\R/\Z$, there is a subsequence  converging to some angle $\s$. As the Julia set is connected, no singular value is escaping, hence the ray $g_{\s}$ of angle $\s$ is well defined for all potentials $t>0$.
Landing of $g_{\s}$ at $\alpha$ follows from Property~\ref{estimates} together with Lemma~\ref{landing}.
 \end{proof}

 To prove periodicity of the landing ray  constructed in Theorem~\ref{Landing of a ray} we   use  the following   lemma.

%

\begin{lem}\label{Rotation sets}
 Let $(X,d)$ be a compact metric space, $f: X\ra X$ be a  locally expanding map and let $A\subset X$ be a closed invariant subset. 
If $f: A\rightarrow A$ is invertible then $A$ is finite.  
\end{lem}

\begin{proof}
    Since  $A$ is compact and $f: A \ra A$ is invertible, it is a homeomorphism,
so $f^{-1} : A \ra A$ is uniformly continuous. 
Let us take $\eps>0$ from  the definition of the local expanding property.   
Then there exists $\de\in (0, \eps)$  such that if $d(x, y) < \de$ for some points 
$x, y\in A$ then $d (f^{-1} x, f^{-1} y) < \eps$. 
Moreover, by the local expanding property
$$
  d(f^{-1} x , f^{-1} y ) \leq \rho^{-1} d(x, y) < \rho^{-1} \de < \de.
$$
Iterating this estimate, we obtain
$$
    d(f^{-n} x , f^{-n} y ) \leq \rho^{-n } d(x, y) < \rho^{-n } \de \to 0 \quad as \ n\to \infty. 
$$
Let us now cover $A$ by $N$ $\de$-balls $B_i\subset A$.  Then the last estimate shows that $A$ is covered by $N$ $(\rho^{-n} \de)$-balls $f^{-n}(B_i)$, which  implies     that $A$ contains at most $N$ points.  
\end{proof}

 \begin{prop}\label{Periodicity of landing ray}
Any  dynamic ray $g_{\s}$ obtained from  the construction of Theorem~\ref{Landing of a ray} is periodic.
 \end{prop}
 \begin{proof}
Let $\BB:=\{\sn\}_{n\in\N}$ be the set of  angles of the rays $\psi^n(g_{\szero})$ constructed in the proof of Theorem~\ref{Landing of a ray}, and $\AA$ be their limit set defined as 
\[\AA:=\{\s\in\Sigma_D: \s_{n_k}\ra\s \text{ for } \s_{n_k}\in \BB, n_k\ra\infty\}.\] 
The set $ \AA$ is closed and forward invariant by definition, and for any $\s\in\AA$ the ray $g_{\s}$ lands at $\alpha$ by the construction of Theorem~\ref{Landing of a ray}.  In order to use Lemma~\ref{Rotation sets}, we begin by showing that $\sigma:\AA\ra\AA$ is a homeomorphism. 
We first show surjectivity of $\sigma$. For any $ \s_{n+1}\in\BB, \sigma \s_{n+1}=\sn$ by definition of $\sn$. If $\s\in\AA$,   there is a sequence $\s_{n_k}\in\BB$ converging to  $\s$; then the sequence $\s_{n_k+1}\in\BB$ has at least one accumulation point $\tilde{\s}\in\AA$ and since $\sigma\s_{n_k+1}=\s_{n_k}$ for all $n_k$, {by continuity} of $\sigma$ we have that $\sigma\tilde{\s}=\s$. {To show  injectivity of $\sigma|_\AA$ suppose by contradiction that there exist $\s, \s'\in \AA$ such that $\sigma \s=\sigma \s'$. Since $\alpha$ is a repelling fixed point it has a simply connected neighborhood $U$ such that $f:U\ra U'\supset U$ is a homeomorphism, and since both $g_\s$ and $g_{\s'}$ land at $\alpha$, there is a  sufficiently small potential $t$ such that $g_\s(t)$ and $g_{\s'}(t)$ are both in $U$. Then  $f(g_\s(t))=f(g_{\s'}(t))=g_{\sigma\s}(F(t))$ contradicting injectivity of $f$ on $U$.} 
The claim then follows from  Lemma~\ref{Rotation sets}.
\end{proof}

The next lemma can be found in \cite{Mi}, Lemma 18.12. The proof holds also in the exponential case; see at the end of Section~\ref{Dynamic rays in the exponential family} for more details.

\begin{lem}\label{One for all}
If a periodic ray lands at a repelling periodic point $z_0$, then only finitely many rays land at $z_0$, and these rays are all periodic of the same period. 
\end{lem}

\begin{cor}
All the rays landing at  a repelling fixed point are periodic.
\end{cor}
This concludes the proof of Theorem A in the polynomial case.

\subsection{Proof of Theorem B in the polynomial case}
In this subsection we prove Theorem B in the polynomial case and we show that under an additional combinatorial condition, every point in a hyperbolic set is the landing point of only finitely many dynamic rays.

\begin{proof}[Proof of Theorem B] Let $\Lambda$ be a hyperbolic set. 
Up to taking an iterate of $f$, 
 we can assume that there is a  $\delta$-neighborhood $U$ of $\Lambda$ such that $|f'(x)|>\eta >1$ for all  $x\in U$. 

Fix some $x_0$ in $\Lambda$, and let us construct a dynamic ray landing at $x_0$. Let $x_n:=f^n(x_0)$, $B'_n:= B_\delta(x_n)$, $B_n:= B_{\delta/\eta }(x_{n})$.
Observe that for each $n$ there is a branch $\psi$ of $f^{-1}$ such that  $\psi(B'_n)\subset B_{n-1}$. We  refer to $\psi^m$ as the composition of such branches, mapping $x_n$ to $x_{n-m}$.
Let $\epsilon:=\delta-\delta/\eta$, and let $t_\epsilon$ be such that the length of fundamental domains starting at $t<t_\epsilon$ is smaller than $\epsilon$ (see Proposition~\ref{Fundamental domains shrinking}).
Let us  define  a family of rays to which we will apply the construction of Theorem~\ref{Landing of a ray}.

Let $t_0<t_\epsilon$ be such that each $B_n$ contains a point of potential $t_0$. For each $n$, let $\AA_{x_n}$ be the family of angles $\s$ such that $g_{\s}(t_0)\in B_n$. By Proposition~\ref{Fundamental domains shrinking}, and because $\dist(\partial B_n,\partial B_n')=\epsilon$, $I_{t_0}(g_{\s})\subset B_n'$ for any $\s\in \AA_{x_n}$. For each $\s\in\AA_{x_n}$, denote by $\psi_*^m \gs$ the ray to which $\psi^m(\gs)(t_0) $ belongs to (see Figure~\ref{B}). For any $m\in\N$ let $t_m:=F^{-m}(t_0)$; following the construction of Theorem~\ref{Landing of a ray}, we obtain that 

\[(\psi_*^m g_{\sn})(t_m, F(t_0))\subset B_{n-m}'\text{ for any }m\leq n,\ \sn\in \AA_{x_{n}}.\] 
Also, 
\begin{equation}\label{pluto}
I_{t_m}(\psi_*^n g_\sn)\subset B\left(x_0,{\frac{\delta}{\eta^m}}\right)\text{ for any } m\leq n,\ \sn\in\AA_{x_n}.
\end{equation}


{Consider now any sequence $\{\sn\}$ of angles such that $\sn\in\AA_{x_n}$, and   let $\an$ be the sequence of angles such that $g_{\a_n}:=\psi_*^n g_\sn$. Observe that $\an\in\AA_{x_0}$, that  $\sigma^n\an=\sn$, and that $I_{t_m}(g_{\a_n})\subset B\left(x_0,{\frac{\delta}{\eta^m}}\right)\text{ for any } m\leq n,\ n\in\N$ by  (\ref{pluto}). Then by  Lemma~\ref{landing}, for any accumulation point $\a$ of the sequence $\{\an\}$ the ray $g_\a$ lands at $x_0$.}
\end{proof}

\begin{figure}[hbt!]
\begin{center}
\def\svgwidth{15cm}
\begingroup%
  \makeatletter%
  \providecommand\color[2][]{%
    \errmessage{(Inkscape) Color is used for the text in Inkscape, but the package 'color.sty' is not loaded}%
    \renewcommand\color[2][]{}%
  }%
  \providecommand\transparent[1]{%
    \errmessage{(Inkscape) Transparency is used (non-zero) for the text in Inkscape, but the package 'transparent.sty' is not loaded}%
    \renewcommand\transparent[1]{}%
  }%
  \providecommand\rotatebox[2]{#2}%
  \ifx\svgwidth\undefined%
    \setlength{\unitlength}{1530.7428246bp}%
    \ifx\svgscale\undefined%
      \relax%
    \else%
      \setlength{\unitlength}{\unitlength * \real{\svgscale}}%
    \fi%
  \else%
    \setlength{\unitlength}{\svgwidth}%
  \fi%
  \global\let\svgwidth\undefined%
  \global\let\svgscale\undefined%
  \makeatother%
  \begin{picture}(1,0.41449878)%
    \put(0,0){\includegraphics[width=\unitlength,page=1]{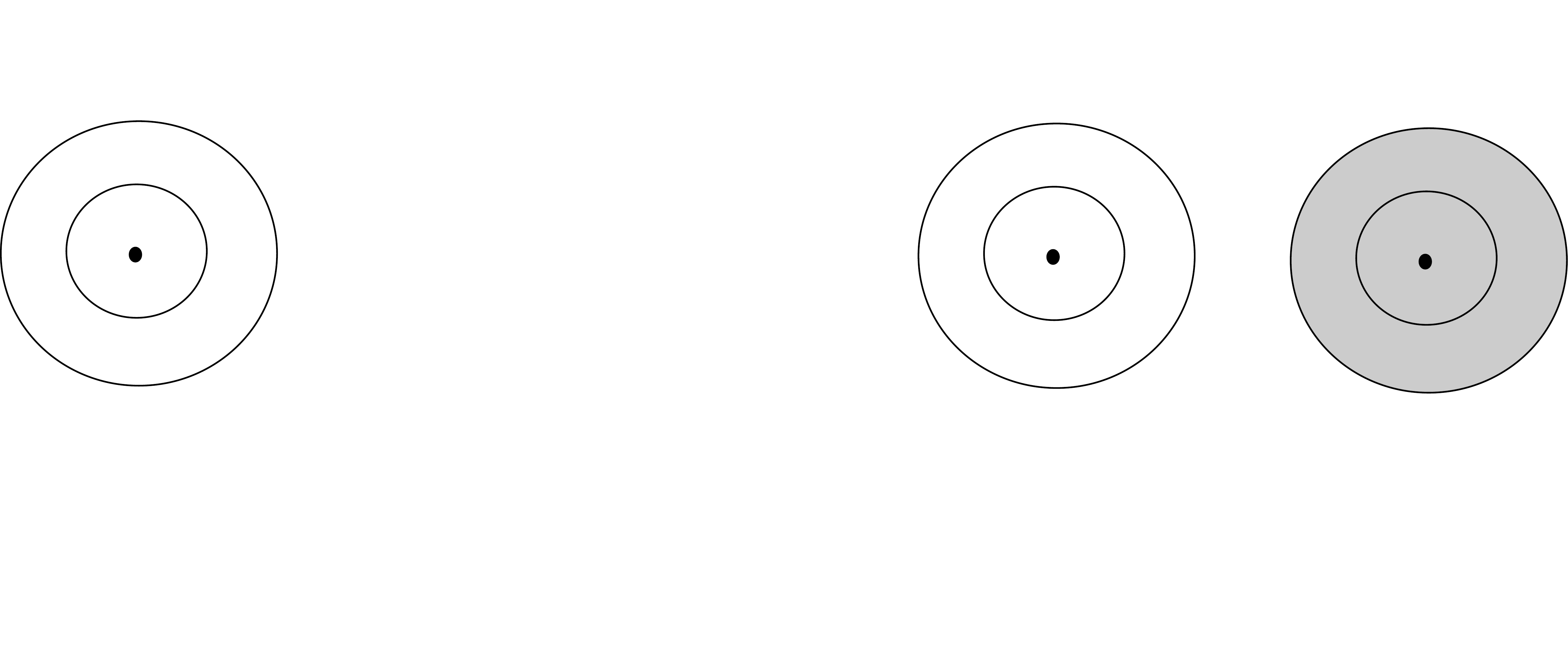}}%
    \put(0.32305509,0.25357081){\color[rgb]{0,0,0}\makebox(0,0)[lb]{\smash{$\ldots$}}}%
    \put(0.07667611,0.22669313){\color[rgb]{0,0,0}\makebox(0,0)[lb]{\smash{$x_0$}}}%
    \put(0.65637126,0.22501875){\color[rgb]{0,0,0}\makebox(0,0)[lb]{\smash{$x_{n-1}$}}}%
    \put(0.8994326,0.22818635){\color[rgb]{0,0,0}\makebox(0,0)[lb]{\smash{$x_n$}}}%
    \put(0.02590709,0.13859399){\color[rgb]{0,0,0}\makebox(0,0)[lb]{\smash{$B'_0$}}}%
    \put(0.63662834,0.13411438){\color[rgb]{0,0,0}\makebox(0,0)[lb]{\smash{$B'_{n-1}$}}}%
    \put(0.87703449,0.13112796){\color[rgb]{0,0,0}\makebox(0,0)[lb]{\smash{$B'_n$}}}%
    \put(0.02881845,0.19409879){\color[rgb]{0,0,0}\makebox(0,0)[lb]{\smash{$B_0$}}}%
    \put(0.6176851,0.19031274){\color[rgb]{0,0,0}\makebox(0,0)[lb]{\smash{$B_{n-1}$}}}%
    \put(0.85131877,0.18838216){\color[rgb]{0,0,0}\makebox(0,0)[lb]{\smash{$B_n$}}}%
    \put(0,0){\includegraphics[width=\unitlength,page=2]{HypSetNew.pdf}}%
    \put(0.7769897,0.39841182){\color[rgb]{0,0,0}\makebox(0,0)[lb]{\smash{$\psi$}}}%
    \put(0.41264746,0.00420547){\color[rgb]{0,0,0}\makebox(0,0)[lb]{\smash{$\psi^n$}}}%
    \put(0,0){\includegraphics[width=\unitlength,page=3]{HypSetNew.pdf}}%
    \put(0.89401234,0.30303879){\color[rgb]{0,0,0}\makebox(0,0)[lb]{\smash{$g_{\s_n}$}}}%
    \put(0.69710319,0.30247333){\color[rgb]{0,0,0}\makebox(0,0)[lb]{\smash{$g_{\s_{n-1}}$}}}%
    \put(0,0){\includegraphics[width=\unitlength,page=4]{HypSetNew.pdf}}%
    \put(0.1222189,0.29575392){\color[rgb]{0,0,0}\makebox(0,0)[lb]{\smash{$g_{\a_n}$}}}%
  \end{picture}%
\endgroup%

\caption{Construction of a landing ray for a point $x_0\in\Lambda$. Here only the pieces of rays $g_\sn$ and $g_{\s_{n-1}}$ between potentials $t_0$ and $F(t_0)$ are shown in $B'_n$ and $B'_{n-1}$ respectively. Similarly, only the piece of ray $g_{\a_n}$ between potential $F^{-n(t_0)}$ and $t_0$ is shown inside $B_0$.} \label{B}
\end{center}
\end{figure}

If  we assume  $f$ to be a unicritical polynomial satisfying some combinatorial conditions, we can show that there are only finitely many rays landing at each $x\in\Lambda$.
We say that two points $z_1,z_2$ are \emph{combinatorially separated} if there is a curve $\Gamma$  formed by two dynamic rays together with a common landing point such that  $z_1,z_2$ belong to different components of $\C\setminus\Gamma$.


\begin{rem}\label{Rotation symmetry}
Let $f$ be a unicritical polynomial of degree $D$. Then there  is a rotational symmetry given by $f(e^{2\pi i /D}z)= f(z)$, implying that two dynamic rays of angles $\s,\s'$ land together if and only if the dynamic rays of angles $\s+j/D, \s'+j/D$ (obtained from the dynamic rays of angles $\s,\s'$ through a rotation of angle $2\pi j /D$) do for $j=1\ldots D-1$. 
 Then by the cyclic order at infinity, two dynamic rays of angles $\s,\s'$ can land together only if $|\s-\s'|\leq 1/D$; otherwise, the two dynamic rays of angles $\s+1/D, \s'+1/D$  would intersect them, giving a contradiction.
 \end{rem}
\begin{prop}\label{Finitely many rays} Let $f(z)=z^D+c$ be a unicritical polynomial of degree $D$, and let  $\Lambda$ be a hyperbolic set. Suppose moreover that either $0$  is accessible or that any $x\in\Lambda$ is combinatorially separated from $0$. Then there are only finitely many dynamic  rays landing at each $x\in\Lambda$.\end{prop}

{ \begin{proof}
Since $\Lambda$ is a hyperbolic set, up to proving the statement for $f^k(\Lambda)$ we can assume that $0\notin \Lambda$. 
 For $x\in\Lambda$, let $\AA_x$ be the set of angles of the rays landing at $x$. 
By Theorem~B, each $\AA_x$ is non empty.
 Near $x$, $f$ is locally a homeomorphism,  so the set $\AA_x$ is mapped bijectively to the set $\AA_{f(x)}$ by the shift map $\sigma$ and there is a well defined inverse $\sigma^{-1}: \AA_{f(x)}\rightarrow \AA_x$. Recall that $\sigma$ is locally expanding by a factor $D$. 
 Let us first assume that $\sigma^{-1}$ is uniformly continuous on the sets $\{\AA_x\}_{x\in\Lambda}$, in the sense that
 there exists an $\eps>0$ such that for any $x\in\Lambda$, and for any  $\a,\a'\in \AA_{f(x)}$ such that $|\a-\a'|<\eps$ we have
\begin{equation}\label{contraction}
  |\sigma^{-1} \a - \sigma^{-1}\a'|< \frac 1D \,|\a- \a'|. 
\end{equation}

\noindent Fix a point $y_0\in\Lambda$ and  let $y_n:=f^n(y_0)$. Consider a finite cover of $\R/\Z$ by $\epsilon$-balls, say $N$ balls, where $\epsilon$ is given by uniform continuity of $\sigma^{-1}$. 
Then for any $n$, $\AA_{y_n}$ is covered by $N$ balls of radius $\epsilon$.
By (\ref{contraction}), their $n$-th preimages have diameter $\epsilon/D^n$ and cover $\AA_{y_0}$. Passing to the limit, we conclude that $\AA_{y_0}$ contains at most $N$ points. 
 
So it is only left to show that $\sigma^{-1}$ is uniformly continuous in the sense of (\ref{contraction}). 
Suppose by contradiction that there is a sequence of points $x_n\in \Lambda$, and two sequences of angles  
$\an, \an'\in \AA_{f(x_n)}$ such that $| \an-\an'|\ra 0$, but 
$|\sigmam \an-\sigmam \an'|\ra \delta>0$. Call $\sn,\sn'$ the angles $\sigmam \an,\sigmam \an'$, and assume for definiteness that $\sn<\sn'$.  Since $\an$ and $\an'$ converge to the same angle and the preimages under $\sigma$ of any angle are spaced at $1/D$ apart, $ \delta= k/D$ for some integer $k\in [1,  D-1]$; since the rays $g_\sn$ and $g_{\sn'}$ land together at the point $x_n$, $k=1$ by Remark~\ref{Rotation symmetry}.
By the $D$-fold symmetry of the Julia set, the rays of angles $\sn+j/D, \sn'+j/D$ for $j=1\ldots D-1$ also land together at the points $e^{2\pi i j/D}x_n$ (see Figure~\ref{C}).

\begin{figure}[hbt!]
\begin{center}
\def\svgwidth{6cm}

\begingroup%
  \makeatletter%
  \providecommand\color[2][]{%
    \errmessage{(Inkscape) Color is used for the text in Inkscape, but the package 'color.sty' is not loaded}%
    \renewcommand\color[2][]{}%
  }%
  \providecommand\transparent[1]{%
    \errmessage{(Inkscape) Transparency is used (non-zero) for the text in Inkscape, but the package 'transparent.sty' is not loaded}%
    \renewcommand\transparent[1]{}%
  }%
  \providecommand\rotatebox[2]{#2}%
  \ifx\svgwidth\undefined%
    \setlength{\unitlength}{1224.6823154bp}%
    \ifx\svgscale\undefined%
      \relax%
    \else%
      \setlength{\unitlength}{\unitlength * \real{\svgscale}}%
    \fi%
  \else%
    \setlength{\unitlength}{\svgwidth}%
  \fi%
  \global\let\svgwidth\undefined%
  \global\let\svgscale\undefined%
  \makeatother%
  \begin{picture}(1,0.90138725)%
    \put(0,0){\includegraphics[width=\unitlength,page=1]{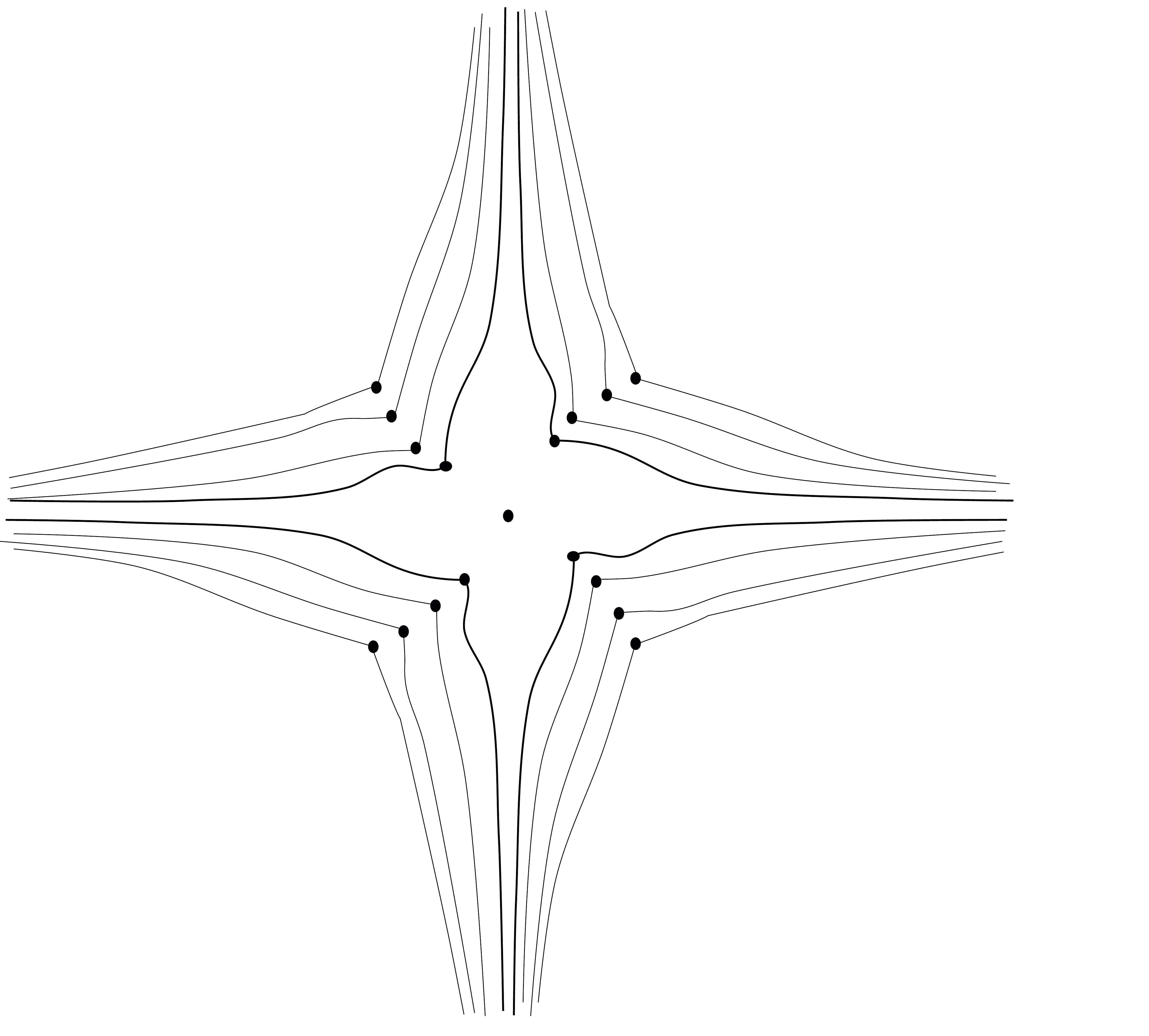}}%
    \put(0.453717,0.4545329){\color[rgb]{0,0,0}\makebox(0,0)[lb]{\smash{$0$}}}%
    \put(-0.05672584,0.1900479){\color[rgb]{0,0,0}\makebox(0,0)[lb]{\smash{}}}%
    \put(0.33779153,0.4424839){\color[rgb]{0,0,0}\makebox(0,0)[lb]{\smash{$V_n$}}}%
    \put(0.45173409,0.38728188){\color[rgb]{0,0,0}\makebox(0,0)[lb]{\smash{$x_n$}}}%
    \put(0.84235469,0.42269233){\color[rgb]{0,0,0}\makebox(0,0)[lb]{\smash{$\s'_n$}}}%
    \put(0.45136835,0.04091894){\color[rgb]{0,0,0}\makebox(0,0)[lb]{\smash{$\s_n$}}}%
    \put(0.84332672,0.47906772){\color[rgb]{0,0,0}\makebox(0,0)[lb]{\smash{$\s_n+1/4$}}}%
    \put(0.46192615,0.8670669){\color[rgb]{0,0,0}\makebox(0,0)[lb]{\smash{$\s_n'+1/4$}}}%
    \put(0,0){\includegraphics[width=\unitlength,page=2]{Symmetry2.pdf}}%
  \end{picture}%
\endgroup%

\end{center}
\caption{\small Illustration to the proof of Proposition~\ref{Finitely many rays} for $D=4$; the region $V_n$ is shaded in gray. Its boundary consists of the rays $g_{\s_n},g_{\s'_n}$ and their rotations by $\pi/2$. The limiting rays are colored in red, and they all either land at $0$ or land at four different points which are rotations of each other by $\pi/2$. 
}\label{C}
\end{figure}

\noindent Altogether, these $D$ pairs of dynamic rays  divide $\C$ into $D+1$ connected component.
Let $V_n$ be the  one that contains the critical point $0$. 
Since for each $i$ we have: 
$$
   (\sn'+j/D)-(\sn+j/D) \ra 1/D\quad \mathrm{as}\ n\to \infty,
$$
 we conclude that 
\begin{equation}\label{sn} 
(\sn+(j+1) /D)-(\sn' + j/D)\ra 0\quad \mathrm{as}\  n\ra\infty.
\end{equation}
  
Let $Q=\bigcap \ov{V_n}$. Then  $0\in Q$ by definition, and $Q\cap\Lambda\neq\emptyset$ because $\Lambda$ is compact and $\ov{V_n}\cap\Lambda\neq\emptyset$
for all $n$. Let $z\in Q\cap \Lambda$.
By (\ref{sn}), there are exactly $D$ limiting rays entering  $Q$.  By symmetry, if these rays land, they either all land at $0$, or they land at $D$ different  points belonging to some orbit of the $2\pi /D$-rotation.  This gives an immediate contradiction with either of our assumptions:
 
\smallskip\noindent $\bullet$  
  Since there are no pairs of rays which could separate $0$ from points of $Q$,
  the point $z\in \Lambda$ is not combinatorially separated from $0$, giving a  contradiction with the first assumption;   

\smallskip\noindent $\bullet$
In the case when $0$ is accessible, this implies that no other accessible point can belong to $Q$; in particular $z\in \Lambda$ is not accessible, contradicting Theorem B.

\noindent This concludes the proof of (\ref{contraction}) and hence of the proposition.
\end{proof}
}

\section{Dynamic rays in the exponential family}
\label{Dynamic rays in the exponential family}

For the remaining two sections, $f(z)= e^z+c$ will be an exponential function, and $J(f)$ will be its Julia set. Any two exponential maps whose singular values differ by $2\pi i $ are conformally conjugate, so we can assume $-\pi\leq \Im c<\pi$.
 Let   $\arg(z)$ be defined on $\C\setminus\R_-$ so as to  take values in $(-\pi,\pi)$, and let   
 \[R:=\{z\in\C:\ \Im z=\Im c,\ \Re z\leq\Re c \}.\]  
\noindent We define a family $\{L_n\}$ of inverse branches for $f(z)=e^z+c$ on $\C\setminus   R$ as
\[L_{n}(w):=\log|w-c|+i\arg(w-c)+2\pi i n. \] 
\noindent Observe that $L_n$ maps $\C\setminus R$ biholomorphically to the strip 
\[ S_n:=\{z\in\C: 2\pi  n  -\pi<\Im z<2\pi n+\pi\}.\]
\noindent Observe also that $|(L_n)'(w)|=\frac{1}{|w-c|}.$ 

Dynamic rays  have been first introduced by Devaney and Krych in \cite{DK} (with the name {\emph{hairs}}) and studied for example in \cite{BD1}. A full classification of the set of escaping points in terms of dynamic rays has been then completed by Schleicher and Zimmer in \cite{SZ1}.
For points whose iterates never belong to $R$, we can consider itineraries with respect to the partition of the plane into the strips $\{S_n\}$, i.e.
 \[\text{itin}(z)=s_0 s_1 s_2\dots\text{ if and only if }f^{j}(z)\in S_{s_j}.\]
\noindent For a point $z$ whose itinerary with respect to this partition exists,  we  refer to it as the \emph{address} of $z$. {The set $\SS$ of all allowable itineraries is called the \emph{set of addresses}; by construction, is is a subset of  $\Z^\N$. 
Addresses as sequences can be characterized also in terms of the growth of their entries. 
  According to the construction, addresses of points cannot have entries  growing faster than iterates of the  exponential function. Let us  use the function $F:t\mapsto e^t-1$ to model real exponential growth. 
A sequence $\s=s_0 s_1 s_2\dots$ is called \emph{exponentially bounded}  if there exists constants $A\geq 1/2\pi$, $x\geq0$ such that for all $k\geq 0$. \begin{equation}\label{address}
|s_k| < A F^{k}(x)\end{equation}
This growth condition turns out to be not only necessary but also sufficient for a sequence to be realized as an address (see \cite{SZ1}), so that $\SS$ also equals the set of all exponentially bounded sequences in $\Z^\N$ (see also Theorem~\ref{Existence of dynamic rays} below).} 

An address is called \emph{periodic} (resp. \emph{preperiodic}) if it is a periodic (resp. preperiodic) sequence. If $\s=s_0 s_1 s_2\dots,$ let $\|\s\|_\infty=\underset{i}\sup |s_i|/2\pi$. We call $\s$ \emph{bounded} if $\|\s\|_\infty<\infty$. For $j\in\Z$ and $\s\in\SS$,  $j \s$ stands for the sequence $j s_0 s_1 s_2\ldots$.
 
Given an address $\s$ we define its \emph{minimal potential} 
\[t_\s:=\inf \left\lbrace t>0:\lim\underset{k\ra\infty}\sup \frac{|s_k|}{F^{k}(t)}=0   \right\rbrace.\]
Observe that if $\s$ is bounded, $t_\s=0$. 

Definition, existence and properties of \emph{dynamic rays} for the exponential family are summarized in  the following theorem (\cite{SZ1}, Proposition 3.2 and Theorem 4.2; the quantitative estimates are taken from  Proposition 3.4).

\begin{thm}[Existence of dynamic rays]\label{Existence of dynamic rays}
Let $f(z)=e^z+c$ be an exponential map such that $c$ is non-escaping, and $K$ be a constant such that $|c|\leq K$. Let $\s=s_0 s_1 s_2\ldots\in \Z^\N$ be a sequence satisfying (\ref{address}), and  let $t_\s$ be the minimal potential for $\s$. Then  there exists a unique maximal injective curve $g_\s: (t_\s,\infty) \rightarrow \C$, consisting of escaping points, such that 
\begin{itemize}
\item[(1)]
$g_\s (t)$ has address $\s$ for $t>x+2\log (K+3)$;
\item[(2)]
$f (g_\s(t))=g_{\sigma \s}(F(t)) $;

\item[(3)] $|g_\s(t)-2\pi i s_0- t-t_\s|\leq 2 e^{-t}(K+2+2\pi |s_1|+2\pi A C)$ for large $t$,\\
      and for a universal constant $C$. 
\end{itemize}
\end{thm}

\noindent The  curve $g_{\s}$ is called the \emph{dynamic ray of address $\s$}. As for polynomials, a dynamic ray is periodic or preperiodic if and only if its address is a periodic or preperiodic sequence respectively. 
Likewise, the \emph{fundamental domain} starting at $t$ for $g_{\s}$ is defined as the arc  $I_t(g_{\s}):= g_{\s}(t,F(t))$. 

\begin{rem}\label{Bound on bounded addresses} {Fix a parameter $c$ which is non-escaping. If $\s$ is bounded, the constant $x$ in Theorem~\ref{Existence of dynamic rays} can be taken arbitrarily small, so  property $(1)$ holds for $t>2\log (K+3)$.
Moreover,  $g_\s(t)$ is approximately straight, i.e. there exists a constant $C'(K, t_\s)$ such that  $|g'_\s(t)- 1|< C'(K,t_\s)$ for large $t$ and  $C'(K,t_\s)$ does not depend on $\s$ if $\s$ is bounded (see Proposition 4.6 in \cite{FS}).
It then  follows from the asymptotic estimates in Theorem~\ref{Existence of dynamic rays}  that for any  $t$ there exists a constant  $B(t, \|\s\|_\infty)$ such that 
\begin{equation}\label{length}
\leucl(I_t(g_\s))\leq B(t, \|\s\|_\infty))\sim e^t-t.
\end{equation}                  
}
\end{rem}

The notion of minimal potential is important to have a parametrization of the rays respecting some kind of transversal continuity.  The following lemma holds, and will play the role of Lemma~\ref{Angle-space continuity} for the exponential family. A proof can be found in \cite{Re3}; we will use the formulation from \cite{Re2}, Lemma 4.7).
\begin{lem}[Transversal continuity]\label{Convergence of rays}
Let $f$ be an exponential map, $\{\sn\}$ be a sequence of addresses, $ \sn\ra \s\in\SS$  such that $t_{\sn}\ra t_{\s}$. Then 
$g_{\sn}\ra g_{\s}$ uniformly on $[t_*, \infty)$ for all $t_*>t_{\s}$. 
\end{lem}

It will be useful to keep in mind the next observation, which moreover  concurs to the validity of Lemma~\ref{One for all} in the exponential case.

\begin{rem}\label{Intersecting rays}
The addresses of two rays $g_\s, g_{\s'}$ landing together cannot  differ by more than one in any entry. Suppose by contradiction that they do: up to taking forward iterates of the two rays (which keep landing together by continuity of $f$) we can assume that $\s=s_0 s_1 s_2\ldots$ and $\s'=s'_0 s'_1 s'_2$ differ by more than one in their first entry. Suppose for definiteness that $s_0<s'_0$.
 By the $2\pi i $  symmetry of the dynamical plane, the rays with addresses $\widetilde{\s}':=(s_0'-1)s'_1 s'_2... $ and $\widetilde{\s}:=(s_0-1)s_1 s_2... $ land together. Since $|s_0-s'_0|\geq 2$, we have that $s_0<\widetilde{s}_0<s'_0$, so by the asymptotic estimates in Theorem~\ref{Existence of dynamic rays} the two curves $\Gamma=\ov{g_\s\cup g_{\s'}}$ and
  $\widetilde{\Gamma}=\ov{g_{\widetilde{\s}}\cup g_{\widetilde{\s}'}}$ would intersect  contradicting the fact that dynamic rays are disjoint.
\end{rem}


We now give  a sketch of the  proof of Lemma~\ref{One for all} in the exponential case, referring to the proof of Lemma 18.12 in\cite{Mi} for polynomials.
{\begin{proof}[Proof of Lemma~\ref{One for all} in the exponential case]
Let $z_0$  be a repelling periodic point which is the landing point of a periodic dynamic ray $g_\s$, and let $\AA$ be the set of addresses of the rays landing at $z_0$. Since $\|\s\|_\infty$ is bounded, by Remark~\ref{Intersecting rays} the norm of all addresses in $\AA$ is also bounded by some constant $M>0$.  Then $\AA\subset \Sigma_D$ for $D=2M+1$, where $\Sigma_D$ is represented as $\Sigma_D=\{-M,\ldots,M\}^{\N}$. The proof of Lemma 18.12 in \cite{Mi} can then be repeated.
\end{proof}
}
\section{Accessibility for exponential parameters with bounded postsingular set}
\label{Accessibility for exponential parameters with bounded postsingular set}

%
%
%
%
%
%

\subsection{Statement of theorems and some basic facts}

From now on we will consider an exponential function $f$ with bounded postsingular set, which implies that the singular value is non-recurrent; this excludes the presence of Siegel disks (see Corollary 2.10, \cite{RvS}).
The strategy used for the proof of Theorem A and B for polynomials can be extended to the exponential family  to prove the following theorems.

 The first theorem is Theorem A stated for the exponential family.

\begin{thm}[Accessibility of periodic orbits for non-recurrent parameters]\label{Non-recurrent parameters}
Let $f$ be an exponential map with bounded postsingular set. Then any repelling periodic point is the landing point of at least one and at most finitely many dynamic rays, all of which are periodic of the same period.
\end{thm}
\begin{cor}
For Misiurewicz parameters,  the postsingular periodic orbit and hence the singular value are accessible.
\end{cor}
 For hyperbolic, parabolic and Misiurewicz parameters Theorem~\ref{Non-recurrent parameters}  has been previously proven in ~\cite{SZ2}. The next theorem is Theorem B stated for the exponential family, with the  additional property that the set of  rays landing on a hyperbolic set have uniformly bounded addresses.

\begin{thm}[Accessibility of hyperbolic sets]\label{Accessibility of hyperbolic sets exp}
Let $f$ be an exponential map with bounded postsingular set, 
and $\Lambda$ be a hyperbolic set. Then any point in $\Lambda$ is accessible; moreover, the dynamic rays landing at  $x\in \Lambda$ all have uniformly bounded addresses.
\end{thm}

\begin{rem}
The family of rays constructed in Theorem~\ref{Accessibility of hyperbolic sets exp} form a lamination. Continuity of the family of rays on compact sets is a consequence of Lemma~\ref{Convergence of rays}. Continuity up to the endpoints follows from the estimates in (\ref{prelam}).
\end{rem}

We also prove that there are only finitely many rays landing at each $x\in\Lambda$.

\begin{prop}\label{Finiteness exp}
Let $f$ be an exponential map with bounded postsingular set, and $\Lambda$ be a hyperbolic set. Then there are only finitely many dynamic rays landing at each $x\in\Lambda$; more precisely there exists a constant $N_0$ such that for each $x\in\Lambda$ there are at most $N_0$ rays landing at $x$.
\end{prop}

In the case in which the postsingular set is bounded and contained in the Julia set, Rempe and van Strien (\cite{RvS}, Theorem 1.2) have shown that it is hyperbolic (see also \cite{MS}, Theorem 3, for a different perspective). Together with Theorem~\ref{Accessibility of hyperbolic sets exp}, this implies accessibility of the postsingular set.

\begin{cor}\label{Accessibility of postsingular set}
Let $f$ be an exponential map with bounded postsingular set.  Then any point in the postsingular set is accessible.
\end{cor}

We will prove Theorem~\ref{Non-recurrent parameters}, Theorem~\ref{Accessibility of hyperbolic sets exp} and Proposition~\ref{Finiteness exp} in Section~\ref{Proof of accessibility theorems}. 
Like in the polynomial case, the strategy is  to first prove a uniform bound on the length of fundamental domains $I_T(g_\a)$ for some fixed $T$ and some specific family of addresses, then to translate this into a uniform shrinking for fundamental domains $I_t(g_\a)$ as $t\ra 0$, and finally to study the local dynamics near a repelling periodic orbit. 
The main difficulties compared to the polynomial case are to find an   analogue of Proposition~\ref{Fundamental domains shrinking}, and to show that the dynamic  rays obtained by pullbacks near the repelling fixed point   have uniformly bounded addresses.

In the following, let $\PP(f)$ be the postsingular set, and $\Omega:=\C\setminus  \PP(f)$.
As $\PP(f)$ is  forward invariant, $\Omega$ is backward invariant, i.e. $f^{-1}(\Omega)\subset \Omega$.

\begin{prop}\label{Connected Omega}
If $\PP(f)$ is bounded, $\Omega$ is connected.
\end{prop}

\begin{proof}
As $\PP(f)$ is bounded, there are no Siegel disks, hence either $J(f)=\C$ or $f$ is parabolic or hyperbolic.
In the last two cases $\PP(f)$ is a totally disconnected set and the claim follows. If $J(f)=\C$, consider the connected components $V_i$ of  $\Omega$; as $\PP(f)$ is bounded, there is only one unbounded $V_i$.  On the other side, by density of  escaping points, each $V_i$ contains escaping points; as dynamic rays are connected sets, each $V_i$ has to be unbounded, hence there is a unique connected component.
\end{proof}

As $\Omega$ is connected, and omits at least three points because $c$ cannot be a fixed point, it admits a well defined hyperbolic density $\rho_\Omega$. 
As $\PP(f)$ is bounded, we have

\begin{equation}\underset{|z|\ra\infty}\lim\frac{\rhohyp(z)}{\rhoeucl(z)}=0. \label{hyperbolic limit}\end{equation}


\noindent{To see this, observe that, since $\PP(f)$ is bounded,  $\rhohyp$ is bounded from above by the hyperbolic density of the set $\Omega':=\C\setminus \D_R$, for some sufficiently large $R$. The latter density $\rho_{\Omega'}$ can be computed directly by using that $\Omega'$ is the image under the exponential map of the right half plane $\{z\in\C: \Re z> \log R\}$. This yields $\rhohyp\leq \frac{1}{|z|}$ as $z\ra\infty$ giving (\ref{hyperbolic limit}).} 



\subsection{Bounds on fundamental domains for exponentials}\label{Bounds on fundamental domains for exponentials}

In this section we prove a uniform bound on the length of fundamental domains for an appropriate family of dynamic rays (Proposition~\ref{Fundamental domains shrinking for exponentials}). 
Observe first that since $\PP(f)$ is bounded, by the asymptotic estimates in Theorem~\ref{Existence of dynamic rays} for any ray $g_\s$ there exists $T$ sufficiently large such that all inverse branches of $f^n$ are well defined and univalent in a neighborhood of $g_\s(T,\infty)$.
 
\begin{prop}[Bounded fundamental domains for exponentials]\label{Bounded fundamental domains for exponentials}
Let $\gst$ be a  dynamic ray with bounded address  and let $C$ be a {sufficiently large}  positive constant. Let
 \[\AA:=\{\a\in\SS: \sigma^m(\a)=\stilde \text{ for some } m\geq0\}, \]
 and let $\{g_\a\}_{\a\in \AA}$ be the collection of pullbacks of $\gst$. Then there exists $T$ such that for all $t>T, \a\in \AA$:             
\begin{itemize}
\item[(P1)] If $\a=a_m ...a_2 a_1\stilde,\text{ then } g_\a(t)=L_{a_m}\circ\dots\circ L_{a_1}\gst (F^m(t))$;        
\item[(P2)] $\Re g_\a(t)>C$;           
\item[(P3)] $\leucl(I_t(g_\a))\leq B(t)$, with $B(t)$ independent of $\a$. 
\end{itemize}
\end{prop}

For any positive constant $C$ let us denote by $\H_C$ the right half plane \[ \H_C:=\{z\in\C:\Re z>C\}.\]
Proposition~\ref{Bounded fundamental domains for exponentials} is a consequence of  the following proposition.

\begin{prop}[Branches of the logarithm]\label{Branches of the logarithm}
  Let $\gst$ be a dynamic ray with bounded address and fix  a small $\epsilon>0$. {Let $C>\max(2,
\Re c+4, \frac{8\pi^2}{\epsilon})$, and such that $\PP(f)\cap \H_C=\emptyset$.}  Then there exists $T>0$ such that for any $m>0$, for any $z=\gst(t)$ with  $t>F^m(T)$ and for any finite sequence $a_m\ldots a_1$  we have the following properties:
\begin{enumerate}
\item[$(1_m)$] $\Re L_{a_m}\circ\dots\circ L_{a_1}(z)-\Re c\geq \Re L^m_0(z)-\Re c-\epsilon/2> C$
\item[$(2_m)$] $|L_{a_m}\circ\dots\circ L_{a_1}(z)-c|\geq|L^m_0(z)-c|-\epsilon$.
\end{enumerate} 
Moreover, for some $C'>0$
\begin{equation}\label{Derivative}|(L_{a_m}\circ\dots\circ L_{a_1})'(z)|\leq e^{\epsilon C'}|(L^m_0)'(z)|.\end{equation}
\end{prop}

\begin{proof}
Denote by $0^m$ the finite sequence formed by $m$ zeroes,  and let 
\[\AA'=\{\a\in \AA: \a=0^m \stilde,\, m\in\N\}.\]
As $\|\stilde\|_\infty<M$ for some $M$, we have $\|\a\|_\infty<M$ for any  $\a\in \AA'$. 
It follows from $(3)$ in Theorem~\ref{Existence of dynamic rays}    that for any $\epsilon$ there exists $T_\epsilon$ such that 
\begin{equation}  \label{Bounds on Aprime}
|g_\a(t)-t|<\epsilon \end{equation} 
\noindent for any $\a\in \AA',\; t>T_\epsilon$.
By Remark~\ref{Bound on bounded addresses}, we have that 

  \begin{equation}\label{Bounds on Aprime address}
   g_{0^m \stilde}(t)=L_0^m\gst(F^m(t))\ \  \text{ for } t\geq 2\log(K+3)
  \end{equation}  
 and that 
\begin{equation}\label{Bounds on Aprime length}
\leucl (I_t(g_\a))<B(t)\sim e^t-t \ \forall a\in\AA'.
\end{equation}

\noindent We first show that $(1_m)$ implies $(2_m)$, and then that $(2_m)$ implies $(1_{m+1})$.
Let $T>T_\epsilon$ be large enough. By $(\ref{Bounds on Aprime})$   for all $t>T$, 

\begin{equation}\label{To the right}
\Re L_0^m(\gst(F^m(t)))>T-\epsilon>C+\Re c \  \text{\  for all $m$}.\end{equation}
 
\noindent For $m=1$, we have that $\Re L_{a_1}(z)=\Re L_0(z)$, so  $(1_1)$ holds by (\ref{To the right}). 

\noindent Now let us show that $(1_m)$  implies $(2_m)$.

\begin{align*}
&|L_0^{m} (z)-c|=\sqrt{|\Re L_0^{m} (z)-\Re c|^2+|\Im L_0^{m} (z)-\Im c|^2}\leq \\
&\leq (\Re L_0^{m} (z)-\Re c)\sqrt{1+\frac{\epsilon}{|\Re L_0^{m} (z)-\Re c|}}\leq\\
&\leq \Re L_0^{m} (z)-\Re c+\frac{\epsilon}{2}\leq \Re L_{a_{m}}\circ\dots\circ L_{a_1}(z)-\Re c+\epsilon\leq \\
&\leq |L_{a_{m}}\circ\dots\circ L_{a_1}(z)-c|+\epsilon. 
 \end{align*}
\noindent To show that  $(2_m)$ implies $(1_{m+1})$ for any $m $, observe that

\begin{align*}
&\Re L_{a_{m+1}}\circ\dots\circ L_{a_1}(z)=\log|L_{a_m}\circ\dots\circ L_{a_1}(z)-c|\geq \\
&\geq\log(|L_0^m(z)-c|-\epsilon)\geq\log|L_0^m(z)-c|-\frac{2\epsilon}{|L_0^m (z)-c|}\geq \\
&\geq\log|L_0^m (z)-c|-\epsilon/2=\Re L_0^{m+1}(z)-\epsilon/2>C+\Re c.
\end{align*}

We now check Equation~(\ref{Derivative}). We need a preliminary observation. Let $w$ be a point with $\Re w >C$ whose image $f(w)$ is in $S_0$. By computing the intersections of $\partial S_0$ with the circle of radius $e^{\Re w}$ centered at $c$ (and using the fact that $|\Im c|\leq \pi$) we obtain 
\begin{equation}
\Re f(w)\geq \Re c+e^{\Re w}/2\geq (\Re w)^2\geq C^2.
\end{equation}
Since if $w:=L_0^{m}(z)$ we have $f^{j}(w)\in S_0$ for $j\leq m-1$, it follows   by induction that 
$$ \Re f^j(w)=\Re L_0^{m-j}(z)\geq (\Re w)^{2^j} \geq C^{2^j}$$ hence 
$$|L_0^{m-j}(z)-c|\geq\Re L_0^{m-j}(z)-\Re c\geq C^{2^j}-\Re c.$$ Let
$C'=2\sum_{0}^{\infty}\frac{1}{C^{2^j}-\Re c}<\infty$, and let $\alpha_i:=|L^i_0(z)-c|$.
 Then by direct computation using Property $(2_m)$ and Taylor expansions we have that 

\begin{align*}
|(L_{a_m}\circ\dots\circ L_{a_1})'(z)|&=\frac{1}{|z-c|}\frac{1}{|L_{a_1}(z)-c|}\cdots\frac{1}{|(L_{a_{m-1}}\circ\dots\circ L_{a_1})(z)-c|}\leq\\
&\leq\frac{1}{|z-c|}\frac{1}{|L_{0}(z)-c|-\epsilon}\cdots\frac{1}{|(L_0^{m-1}(z)-c|-\epsilon}\leq\\
&\leq |(L^m_0)'(z)| \prod_{i=1}^{m-1} \frac{1}{1-\epsilon/\alpha_i}
\leq |(L^m_0)'(z)| e^{-\sum_{i=1}^{m-1} \log({1-\epsilon/\alpha_i})}\leq\\
&\leq|(L^m_0)'(z)|e^{\epsilon C'}.
\end{align*}
\end{proof}

\begin{proof}[Proof of Proposition~\ref{Bounded fundamental domains for exponentials}.]
 All addresses in $\AA$ are bounded, so by Remark~\ref{Bound on bounded addresses}, $g_\a(t)$ has address $\a$ for  $t>2\log(K+3)$, proving Property $(P1)$.   
Take $T$ such that Proposition~\ref{Branches of the logarithm}  holds. 
From Property $(1_m)$, 
$\Re g_\a(t)\geq \Re L_0^m(g_{\s}(F^m(t)))-\epsilon/2> C+\Re c$  proving Property $(P2)$. 

Finally, $\leucl(I_t(g_\a))\leq B(t)$ for all $\a\in \AA'$ by (\ref{Bounds on Aprime length}), so    $\leucl (I_t(g_\a))\leq B'(t)$ for all $\a\in \AA$ and for some $B'$ by (\ref{Derivative}); property $(P3)$ follows.
\end{proof}

Let us now show that the length of fundamental domains $\{I_t(g_\a)\}
_{\a\in\AA}$ shrinks as $t\ra0$ for the family of pullbacks of the ray $g_\stilde$. The next proposition  follows from  classical arguments for normal families, see e.g. \cite{L0}, Proposition 3. We include a proof for completeness.

\begin{prop}[Shrinking under inverse iterates]\label{Shrinking under inverse iterates}
Let $f$ be an exponential map, and $U$ be a simply connected domain not intersecting the postsingular set.  
Let $K$ be any   compact set and let $L$ be another compact set which is compactly contained in $U$ and such that $L\subset J(f)$.  Denote by  $\{f_\lambda^{-m}\}$ the family of branches  of $f^{-m}$ such that  $f_\lambda^{-m}(L)\cap K\neq\emptyset$. 
 Then 
\[\diam(f_\lambda^{-m}(L))\ra0 \text{ as } m\ra\infty\]
uniformly in $\lambda$.
\end{prop}

\begin{proof}

As $U$ is simply connected, and does not intersect the postsingular set, inverse branches are well defined 
on $U$.
Suppose by contradiction that there is $\epsilon>0,\ m_k\ra\infty,$ and branches $f^{-m_k}_{\lambda_k}$ of $f^{-m_k}$ such that

\begin{itemize} \item[(1)] $\diam_\eucl(f_{\lambda_k}^{-m_k}(L))>\epsilon$ for any $m_k, \lambda_k$;
\item[(2)] 
$f_{\lambda_k}^{-m_k}(L) \cap K\neq \emptyset$ for any $m_k, \lambda_k$.
\end{itemize}

By normality of inverse branches, there is a subsequence converging to a univalent function $\phi$, {which is non-constant by $(1)$}. 
By $(2)$ there is a sequence of points $\{x_k\}\in L$ such that $f_{\lambda_k}^{-m_k}(x_k)\in K$, and by compactness of $K$, the $f_{\lambda_k}^{-m_k}(x_k)$ accumulate on some point $y\in K$. 
 As $\phi$ is not constant, there is  a neighborhood $V$ of $y$  and infinitely many $m_k$ such that $ f^{m_k} (V)\subset U$, contradicting the fact that $y\in J(f)$.
\end{proof}

\begin{prop}[Fundamental domains shrinking for exponentials]\label{Fundamental domains shrinking for exponentials}
Let $\gst$ be a dynamic ray with bounded address, $\{g_\a \}_{\a\in\AA}$  be its family of pullbacks as defined in Proposition~\ref{Bounded fundamental domains  for exponentials}.
 Given an  $\epsilon>0$ and a compact set $K$ there exists  $t_\epsilon=t_\epsilon(K)$ such that $\leucl(I_t(g_\a))<\epsilon$ whenever $t < t_\epsilon$ and $I_t(g_\a)\cap K\neq\emptyset$.
\end{prop}

\begin{proof}
Let  $\Omega=\C\setminus \PP(f)$, and consider the hyperbolic metric on $\Omega$. Let $C>0$ such that Proposition~\ref{Bounded fundamental domains for exponentials} holds and such that $\PP(f)\cap \H_C=\emptyset$. 
Let  $T$ be as in Proposition~\ref{Bounded fundamental domains for exponentials} so that $I_t(g_\a)\subset \H_C$ for any $\a\in \AA$, $t>T$. Inverse branches  of $f^{n}$ are well defined on the family  $\{I_T(g_\a)\}$ because it is contained in a right half plane not intersecting the postsingular set. 

We first show that, for fundamental domains which are pullbacks of the family  $\{I_T(g_\a)\}$ and which intersect $K$, in order to have small Euclidean length it is sufficient to have small hyperbolic length with respect to the hyperbolic metric on $\Omega$.  
The arcs $\{g_\a (T,F^2(T))\}$ have uniformly bounded Euclidean length by $(P3)$ in Proposition~\ref{Bounded fundamental domains for exponentials}, so they have uniformly bounded hyperbolic length because they are contained in $\H_C$  and $\H_C\cap\PP(f)=\emptyset$ (together with the asymptotic estimates in (\ref{hyperbolic limit})). Then by the Schwarz Lemma the hyperbolic length of the arcs in the families $\{I_t(g_\a)\}_{t<T}$ is bounded uniformly as well.
Since $K$ is compact all fundamental domains which intersect $K$ have also uniformly bounded Euclidean length, so there exists a compact set $K'\supset K$ such that $I_t(g_\a)\subset K'$ whenever $I_t(g_\a)\cap K\neq\emptyset$. 
By compactness of  $K'$ there exists $\epsilon'$ such that $\leucl(\gamma)<\epsilon$ for any curve $\gamma\subset (K'\cap\Omega)$ with $\lhyp(\gamma)<\epsilon'$. 

So it remains to show that there exists $t_\epsilon$ such that, for all $\a\in\AA$ and 
$t<t_\epsilon$, $\lhyp(I_t(g_\a))<\epsilon'$ whenever  $I_t(g_\a)\cap K\neq\emptyset$. 
Since the length  $\leucl(I_T(g_\a))$ is bounded for every $\a\in \AA$ by Proposition~\ref{Bounded fundamental domains for exponentials}, and the family $\{I_T(g_\a)\}$ is contained in $\H_C$,  by the asymptotic estimates in (\ref{hyperbolic limit}) there exists a closed  disk $D$  of sufficiently large radius such that $\lhyp(I_T(g_\a))<\epsilon'$ for any $\a\in\AA$ with $I_T(g_\a)\not\subset D$. For the pullbacks of any such $\a$ the claim then follows by the Schwarz Lemma (and in fact, for such $\a$ $t_\epsilon=T$).  
 Let us now consider any $\a\in \AA$ such that $I_T(g_\a)\subset D\cap\H_C$. By Proposition~\ref{Shrinking under inverse iterates} there is $n_\epsilon$ such that $\diam_{\text{eucl}}f^{-n}(I_T(g_\a))<\epsilon$ for any $n>n_\epsilon$ and any branch such that   $f^{-n}(I_T(g_\a)\cap K\neq\emptyset$; the claim then holds for any $t<t_\epsilon=F^{-n_\epsilon}(T)$.
\end{proof}

\subsection{Proof of Accessibility Theorems}
\label{Proof of accessibility theorems}

In this Section we prove Theorems~\ref{Non-recurrent parameters} and 
\ref{Accessibility of hyperbolic sets exp} as well as Proposition~\ref{Finiteness exp}. To ease the following proofs we introduce the proposition below.

\begin{prop}\label{Families of bounded addresses} 
Let $K$ be a compact set, $\gz$ be a dynamic ray with bounded address, $\AA$ as in Proposition~\ref{Bounded fundamental domains for exponentials}, and $\AA'$ be any subset of $\AA$ such that for some $t_0>0$ and for all $\a\in\AA'$ the following properties hold:
\begin{itemize}
\item[(1)] $g_\a(t_0)\in K$;
\item[(2)] If $\a\neq\szero$, then $\sigma\a\in\AA'$.
\end{itemize}
 Then there exists $M>0$ such that for all  $\a\in\AA'$, $ \|  \a \|_\infty<M$. The constant $M$ depends on $K, t_0$ and $\s_0$.
\end{prop}

\begin{proof}
Let $T$ be as in Proposition~\ref{Bounded fundamental domains for exponentials}, and let $N$ be such that $F^N(t_0)>T$. The set $f^N(K)$ is compact, so there exists $M'>\|\szero\|_\infty$ such that 
$|\Im z|<2\pi M'$ for all $z\in f^N(K)$. 
For any $\a\in\AA'$,  $g_\a(t_0)\in K$, so $g_{\sigma^N\a}(F^N(t_0))\in f^N(K)$ and  by Property $\emph{(P1)}$ in Proposition~\ref{Bounded fundamental domains for exponentials}, the first entry of $\sigma^N\a$ is bounded by $M'$. From {\cal{(2)}}, we get that $\|\sigma^N\a\|_\infty<M'$ for any $\a\in\AA$.

We now show that the points ${g_\a(F^N(t_0))}$ are also contained in finitely many of the strips $S_n$; then, by Property $\emph{(P1)}$ in Proposition~\ref{Bounded fundamental domains for exponentials}, the first entry of $\a$ is bounded by some $M$ for all $\a\in\AA$, and the claim follows by $(2)$. 
Using (\ref{length})  it follows that  
 $\leucl(g_{\sigma^N\a}(F^N(t_0), F^{2N}(t_0)))\leq B$ for some $B>0$, and for all $\a\in\AA'$. Since $|(f^N)'|$ is bounded on $K$ because $K$ is compact, there exists $B'>0$ such that $\leucl(g_{\a}(t_0, F^{N}(t_0)))\leq B'$ for all $\a\in\AA'$ so that the points ${g_\a(F^N(t_0))}$ are also contained in finitely many of the strips $S_n$ (since they are a finite distance away from the compact set $K$).
\end{proof}

We now prove Theorem~\ref{Non-recurrent parameters}.

\begin{proof}[Proof of Theorem~\ref{Non-recurrent parameters}]
 Let us first assume that the repelling periodic point under consideration is a  fixed point  $\alpha$.
Let $U,U', \psi$ and $\epsilon$ be as in the proof of Theorem~\ref{Landing of a ray}. 
Let $\gst$ be a dynamic ray with bounded address, $\{g_\a\}_{\a\in \AA}$ be its family of pullbacks and $C,T$ be such that Proposition~\ref{Bounded fundamental domains for exponentials} holds. 
Let $K:=\overline{U'}$ and $t_\epsilon$ be given by Proposition~\ref{Fundamental domains shrinking for exponentials} so that, for any $t<t_\epsilon$ and for any $I_t(g_\a)$ intersecting $\overline{U'}$, we have that 
$\ell_\eucl(I_t(g_\a))<\epsilon$.

Consider a dynamic ray $g_\szero$  such that $g_\szero(t_0)\in U$ for some $t_0<t_\epsilon$. To show that such a point exists, observe that since $U\cap J(f)\neq\emptyset$  there exists a sufficiently large $N_0$ with  $f^{N_0}(U)\cap I_{t_\epsilon}(\gst)\neq\emptyset$,  so there exists $y\in  I_{t_\epsilon}(\gst)$ such that $f^{-{N_0}}(y)\in U$ giving the desired $g_\szero(t_0)$.

Now we can use the inductive construction from Theorem~\ref{Landing of a ray} to obtain a sequence of dynamic rays $g_\sn$  such that the arcs  $\gamma_n:=g_\sn(t_n, F(t_0))$ are well defined and satisfy properties \emph{1-3} in the proof of  Theorem~\ref{Landing of a ray}. 
By Proposition~\ref{Families of bounded addresses}, $\|\s_n\|_\infty$ is uniformly bounded, so there exists at least one limiting address $\s$ for the addresses $\sn$; by  Lemma~\ref{landing} and Lemma~\ref{Convergence of rays}, the ray $g_\s$ lands at $\alpha$. 

In the case of a repelling periodic point of period $p>1$, the construction can be repeated using as fundamental domains the arcs between potential $t_0$ and $F^p(t_0)$.
 
The proof of periodicity of the landing rays is the same as the proof for polynomials, using the fact that the family of addresses $\{\sn\}$ is uniformly bounded hence Lemma~\ref{One for all} holds. 
\end{proof}

{The proof of accessibility of hyperbolic sets for an exponential map in the exponential case (Theorem~\ref{Accessibility of hyperbolic sets exp}) is a bit more complicated than in the polynomial case, essentially because Proposition~\ref{Fundamental domains shrinking} holds for all fundamental domains while Proposition~\ref{Fundamental domains shrinking for exponentials} holds only for the family of pullbacks of a given ray with bounded address. So, when constructing a ray landing at a point $x_0$,  we will need to consider fundamental domains all of which belong to the family of pullbacks of some initial ray with bounded address. To do this, we will consider a 
subsequence of the  iterates of $x_0$ all of which have near by a  fundamental domain $I_{t_0}(g_\szero)$(with sufficiently small potential $t_0$) of some ray $g_\szero$ with bounded address; then, we will pull back $I_{t_0}(g_\szero)$ and use the inductive construction from Theorem~\ref{Landing of a ray} in order to obtain arcs of rays near $x_0$ satisfying (\ref{prelam1}) and (\ref{prelam}).

\begin{proof}[Proof of Theorem~\ref{Accessibility of hyperbolic sets exp}]

Up to taking an iterate of $f$, we can assume that there is a  $\delta$-neighborhood $U$ of $\Lambda$ such that $|f'(x)|>\eta >1$ for all  $x\in U$. Let $\gst$ be a dynamic ray with bounded address and  $\{g_\a \}_{\a\in \AA}$ be its the family of pullbacks. 
For $\epsilon:=\delta-\delta/\eta$, let $t_\epsilon$ be such that  $\ell_{\eucl}(I_t(g_\a))<\epsilon$ for $\a\in\AA$ and   $t<t_\epsilon$ whenever $I_{t}(g_\a)\cap\ov{U}\neq\emptyset$ (see Proposition~\ref{Fundamental domains shrinking for exponentials}).

Let $t_0<t_\epsilon$ be such that for any $x\in\Lambda$, $B_{\delta/\eta}(x)$ contains  $g_\a(t_0)$ for some $\a\in\AA$ depending on $x$. To show that such a $t_0$ exists, consider a finite covering $\DD$ of $\Lambda$ by balls of radius $\frac{\delta}{3\eta}$. For any $D\in\DD$, there is  $N_D$ such that $f^{N_D}(D)\supset g_\stilde (0,t_\epsilon)$, hence for each $D$ there is some $\a\in\AA$ such that $g_\a(0, F^{-N_D}(t_\epsilon))\subset D$. 
By letting $N=\max{N_D}$, we have that $g_\a(0, F^{-N}t_\epsilon)\subset D$ for any $D$, hence that each $B_{\delta/\eta}(x)$ contains   $g_\a(F^{-N}(t_\epsilon))$ for some $\a\in\AA$.

Now fix $x_0$ in $\Lambda$, and let us construct a dynamic ray landing at $x_0$.  Let $x_n:=f^n(x_0)$, $B'_n:= B_\delta(x_n)$, $B_n:= B_{\delta/\eta }(x_{n})$.
For each $n$ there is a branch $\psi$ of $f^{-1}$ such that  $\psi(B'_n)\subset B_{n-1}$. Take a disk $D\in\DD$ containing infinitely many $x_n$, say  $\{x_n\}_{n\in\NN}\in D$ for some infinite set $\NN\subset \N$. Observe that for any such $n$, $D\subset B_n$. Let $\szero\in\AA$ be such that $g_\szero(t_0)\in D$; by Proposition~\ref{Fundamental domains shrinking for exponentials}, $\leucl(I_{t_0}(g_\szero))<\epsilon$, so for all $n\in \NN$ we have that  $I_{t_0}(g_\szero)\subset B_n'$.

 For $n\in\NN$, let $\psi^n$ be the branch of $f^{-n}$ mapping $x_n$ to $x_0$;  $\psi^n$ can be extended analytically to $B'_n$ and hence to $I_{t_0}(g_\szero)$. Let $g_\sn$ be the sequence of rays containing  $\psi^n (I_{t_0}(g_\szero))$. Let $t_n=F^{-n}(t_0)$.  
Following the inductive construction from the proof of Theorem~\ref{Landing of a ray},  the arcs $g_\sn(t_n,F(t_0))$ satisfy the following properties:

\begin{align}\label{prelam1}
& g_\sn (t_n,F(t_0))\subset B_0' \text{ and }\\
& I_{t_m}(g_\sn)\subset B\left(x_0, \frac{\delta}{\eta^m}\right) \text{ for all $n>m$, $n\in\NN$.} \label{prelam}
\end{align}

\noindent 
Consider now the family of addresses $\sn$ constructed for $x_0$ and the set
\[\widetilde{\AA}_{x_0}=\left\lbrace \a\in\SS:  \a=\sigma^j \sn \text{ for some }j\leq n \text{ and } n\in\NN        \right\rbrace.\]
By definition, if $\a\in\widetilde{\AA}_{x_0}$ then $\sigma\a\in\widetilde{\AA}_{x_0}$ unless $\a=\szero$, and for each $\a\in\widetilde{\AA}_
{x_0}$, $g_\a(t_0)\in\ov{U}$.  It follows  by Proposition~\ref{Families of bounded addresses} that $\|\a\|_\infty<M$ for all $\a\in \widetilde{\AA}_{x_0}$. In particular, $\|\sn\|_\infty<M$, and there is a limiting address $\s$ such that $g_\s$ lands at $x_0$ by Lemma~\ref{landing}.
By finiteness of $\DD$, there exists $M'>0$ such that  the addresses of all rays coming from the construction and landing at $x\in \Lambda$ are  bounded by $M'$. 
Finally, the addresses of any two rays landing  together cannot differ by more than one in any entry (see Remark~\ref{Intersecting rays}), so the address of any dynamic ray (not necessarily coming from this construction) landing at any $x\in \Lambda$ is bounded by $M'+1$.
\end{proof}

Let us conclude by showing that also in the exponential case, there are only finitely many rays landing at each point in  a hyperbolic set, thus proving Proposition~\ref{Finiteness exp}. The proof is very similar to the polynomial case.

\begin{proof}[Proof of Proposition~\ref{Finiteness exp}]

For $x \in \Lambda$, let $\AA_x$ be the set of addresses of the rays landing at $x$. 
By Theorem~\ref{Accessibility of hyperbolic sets exp}, each $\AA_x$ is non-empty.
 Since $f$ is  locally a homeomorphism  near $x$, the set $\AA_x$ is mapped bijectively to the set $\AA_{f(x)}$ by the shift map $\sigma$, so there is a well defined inverse $\sigma^{-1}: \AA_{f(x)}\ra\AA_x$.
 By Theorem~\ref{Accessibility of hyperbolic sets exp}, the norm of the addresses belonging to the set $\AA_x$ is  bounded by a constant $M$, so $\bigcup_{x\in\Lambda}\AA_x\subset \Sigma_D\subset\R/\Z$ for  $D=2M+1$ preserving the dynamics.
 
 

 
 {Assume first that $\sigma^{-1}$ is uniformly continuous in the sense of (\ref{contraction}). Then the proof is the same as in the polynomial case by considering a finite cover of $\Sigma_D$ by $\epsilon$-balls.}
 
Let us now prove (\ref{contraction}) in the exponential setting. Assume by contradiction that there is a sequence of points $x_n\in \Lambda$, and two sequences of addresses  
$\an, \an'\in \AA_{f(x_n)}$ such that $| \an-\an'|_D\ra 0$, but 
$|\sigmam \an-\sigmam \an'|_D\ra\delta\neq0$. Let $ \sn=\sigma^{-1}\an$, $ \sn'=\sigma^{-1}\an'$, and  assume for definiteness that $\s'_n<\s_n$.  Since $| \an-\an'|_D\ra 0$,  the sequences  $\an$ and $\an'$ converge to the same address, say $\a$.  Since  $\{\sigma^{-1}(\a)\}=\{j\a \text{ with } j\in\Z \}$, we have that $\sn$ and $\sn'$ converge to the same sequence up  to the first entry, so that $\delta=k/D$  for some $k\in \N$. Since $g_\sn$ and $g_{\sn'}$ land together, by Remark~\ref{Intersecting rays} we have that $k=1$.
This means that  the dynamic rays of addresses $\s_n=j\a_n$ and $\sn'=(j+1)\a_n'$ belong to the same $\AA_{x_n}$, i.e. land together for all  $j\in \Z$.

These pairs of rays divide $\C$ into infinitely many regions exactly one of which contains a left half plane, and which  we call $V_n$ (see Figure~\ref{E}).
\begin{small}
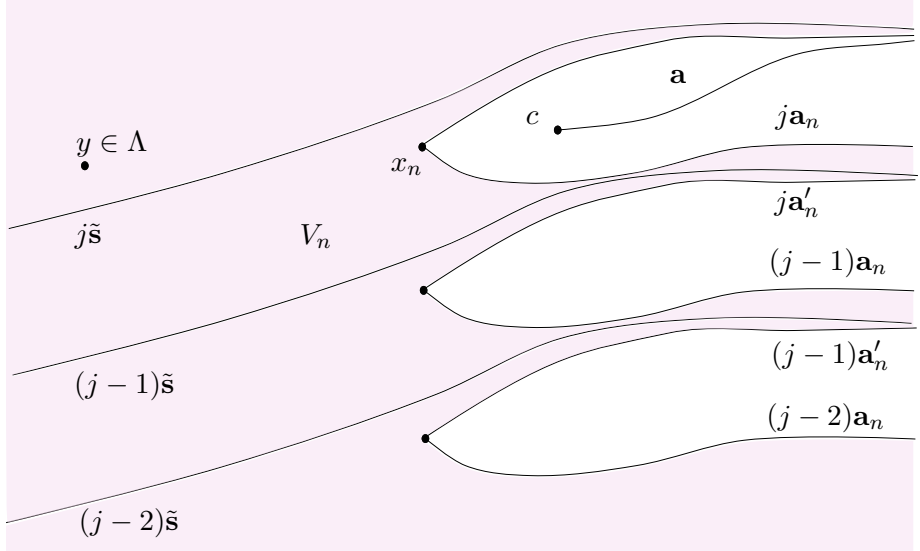
\begin{figure}[hbt!]
\begin{center}
\def\svgwidth{12cm}

\begingroup
  \makeatletter
  \providecommand\color[2][]{
    \renewcommand\color[2][]{}%
  }
  \providecommand\transparent[1]{
    \renewcommand\transparent[1]{}%
  }
  \providecommand\rotatebox[2]{#2}
  \ifx\svgwidth\undefined
    \setlength{\unitlength}{1639.1pt}
  \else
    \setlength{\unitlength}{\svgwidth}
  \fi
  \global\let\svgwidth\undefined
  \makeatother
  \begin{picture}(1,0.60939381)%
    \put(0,0){\includegraphics[width=\unitlength]{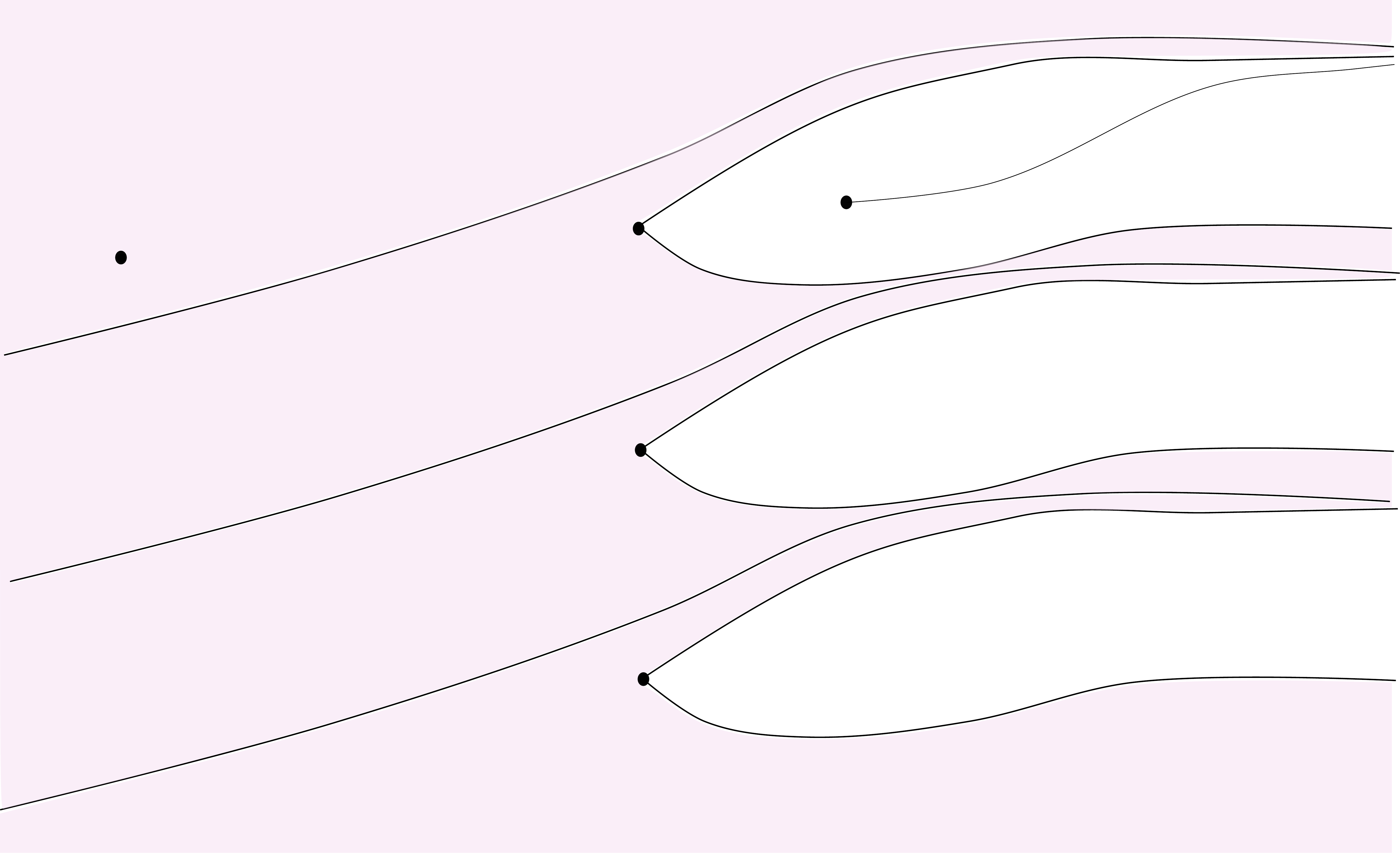}}%
    \put(0.07678578,0.4412187){\color[rgb]{0,0,0}\makebox(0,0)[lb]{\smash{$y\in \Lambda$}}}%
    \put(0.07557497,0.33927803){\color[rgb]{0,0,0}\makebox(0,0)[lb]{\smash{$j \tilde{\s}$}}}%
    \put(0.07482819,0.17348796){\color[rgb]{0,0,0}\makebox(0,0)[lb]{\smash{$(j-1)\tilde{\s}$}}}%
    \put(0.07975845,0.02208029){\color[rgb]{0,0,0}\makebox(0,0)[lb]{\smash{$(j-2)\tilde{\s}$}}}%
    \put(0.32224566,0.33891314){\color[rgb]{0,0,0}\makebox(0,0)[lb]{\smash{$V_n$}}}%
    \put(0.72770487,0.51612741){\color[rgb]{0,0,0}\makebox(0,0)[lb]{\smash{$\a$}}}%
    \put(0.84361207,0.47124986){\color[rgb]{0,0,0}\makebox(0,0)[lb]{\smash{$j\a_n$}}}%
     \put(0.57,0.47124986){\color[rgb]{0,0,0}\makebox(0,0)[lb]{\smash{$c$}}}%
    \put(0.84254144,0.37619707){\color[rgb]{0,0,0}\makebox(0,0)[lb]{\smash{$j\a_n'$}}}%
    \put(0.8365405,0.30998162){\color[rgb]{0,0,0}\makebox(0,0)[lb]{\smash{$(j-1)\a_n$}}}%
    \put(0.83897056,0.20741381){\color[rgb]{0,0,0}\makebox(0,0)[lb]{\smash{$(j-1)\a_n'$}}}%
    \put(0.83386271,0.13870846){\color[rgb]{0,0,0}\makebox(0,0)[lb]{\smash{${(j-2)\a_n}$}}}%
    \put(0.42186804,0.41961005){\color[rgb]{0,0,0}\makebox(0,0)[lb]{\smash{$x_n$}}}%
  \end{picture}%
\endgroup

\end{center}
\caption{\small Illustration to the proof of Proposition~\ref{Finiteness exp}; the true picture is invariant under translation by $2\pi i $. The rays are labeled by their addresses. The region $V_n$ is shaded. }\label{E}
\end{figure}
\end{small}
 
 \noindent Let $Q=\bigcap \ov{V_n}$. Since for each $n\in\N$ there is  $x_n\in\Lambda\cap\ov{V_n}$ and  $\Lambda$ is compact  we have that  $Q\cap \Lambda\neq\emptyset$ and contains at least a point $y\in\C$. Observe that by (3) in Theorem~\ref{Existence of dynamic rays} and since rays do not intersect, a dynamic ray of address $\s$ is contained in $V_n$ if and only if  $j\an'<\s< j\an$ for some $j\in\Z$. Taking the limit as $n\ra\infty$, and recalling that 
 $\lim \a_n=\lim \a_n'=\a$, we have that no dynamic  rays can intersect $Q$ except for the dynamic rays of address $j\a$, $j\in\N$. However, by  Corollary~\ref{Accessibility of postsingular set} the singular value is accessible by some  dynamic ray $g_\stilde$, whose countably many preimages $g_{j\stilde}$ intersect any left half plane: it follows that $\stilde=\a$ and that none of the dynamic rays of address $j\a$ is landing. This contradicts the fact that $\Lambda\cap Q\neq\emptyset$ and that points in $\Lambda$ are accessible by Theorem~\ref{Accessibility of hyperbolic sets exp}. 
\end{proof}

\subsection{A remark about parabolic wakes}
\label{A remark about parabolic wakes}

 Theorem~\ref{Non-recurrent parameters} implies that non-recurrent parameters with bounded post-singular set always belong to parabolic wakes (see \cite{Re1}, Proposition 4 and 5, for the definition of parabolic wakes and the relation between parabolic wakes and landing of rays in the dynamical plane; see also \cite{RS1}).

 There is a combinatorial proof of this fact (see e.g. in \cite{S1}) which can be adapted to the exponential case once Theorem~\ref{Non-recurrent parameters} is known. 
\begin{cor}[Corollary of Theorem~\ref{Non-recurrent parameters}]\label{Non-recurrent parameters in wakes}
A  non-recurrent parameter $c$ with bounded postsingular set for the exponential family is contained in a parabolic wake attached to the boundary of  the period one hyperbolic component $W_0$.
\end{cor}

\begin{proof}[Sketch of proof]
Let $c_0\in W_0$ be a hyperbolic parameter with an attracting fixed point $\alpha(c_0)$. Observe that for any  fixed address $\s$, the ray $g^{c_0}_\s$ lands  at a fixed point $z_{\s}(c_0)$. Let $\gamma$ be a curve joining $c$ to $c_0$ such that all $z_{\s}(c_0)$ can be continued analytically together with their landing rays. Call  $\alpha(c)$  the analytic continuation along $\gamma$ of the attracting fixed point $\alpha(c_0)$, $z_{\s}(c)$ the analytic continuation of $z_{\s}(c_0)$.
   Suppose  that $c$ does not belong to any parabolic wake attached to $W_0$.  By Theorem A, $\alpha(c)$ is the landing point of at least one periodic  ray, which is necessarily fixed (otherwise, there would be at least two periodic rays landing at $\alpha(c)$ determining a parabolic wake). But this gives a contradiction, because for any  fixed address $\s$ the fixed ray $g_\s^c$ lands at the point $z_\s(c)\neq \alpha(c)$.
\end{proof}

\begin{small}

\end{small}
\end{document}

%% file: Tesimacros.tex
\newtheorem{thm}{Theorem}[section]

\newtheorem{cor}[thm]{Corollary}
\newtheorem{lem}[thm]{Lemma}
\newtheorem*{thmA}{Theorem A}
\newtheorem*{thmB}{Theorem B}
\newtheorem*{thmC}{Corollary C}

\newtheorem{prop}[thm]{Proposition}

\theoremstyle{definition}

\newtheorem{rem}[thm]{Remark}

\theoremstyle{remark}

\numberwithin{equation}{section}

\font\nt=cmr7

\def\note#1{}
\def\note#1
{\marginpar
{\nt $\leftarrow$
\par
\hfuzz=20pt \hbadness=9000 \hyphenpenalty=-100 \exhyphenpenalty=-100
\pretolerance=-1 \tolerance=9999 \doublehyphendemerits=-100000
\finalhyphendemerits=-100000 \baselineskip=6pt
#1}\hfuzz=1pt}

\def\be{\begin{equation}}
\def\ee{\end{equation}}

\renewcommand{\epsilon}{\varepsilon}

\newcommand{\ra}{\rightarrow}

\newcommand{\diam}{\operatorname{diam}}
\newcommand{\dist}{\operatorname{dist}}

\newcommand{\eps}{{\epsilon}}

\newcommand{\de}{{\delta}}

\newcommand{\Om}{{\Omega}}

\renewcommand{\AA}{{\cal A}}

\newcommand{\BB}{{\cal B}}

\newcommand{\DD}{{\cal D}}

\newcommand{\FF}{{\cal F}}

\newcommand{\NN}{{\cal N}}

\newcommand{\PP}{{\cal P}}

\renewcommand{\SS}{{\cal S}}

\newcommand{\C}{{\mathbb C}}

\newcommand{\D}{{\mathbb D}}

\newcommand{\N}{{\mathbb N}}

\newcommand{\R}{{\mathbb R}}

\newcommand{\Z}{{\mathbb Z}}

\def\B0{{\mathbf{0}}}

\renewcommand{\Im}{\operatorname{Im}}
\renewcommand{\Re}{\operatorname{Re}}

\newcommand{\ov}{\overline}

\renewcommand{\ra}{\rightarrow}

\catcode`\@=12

\def\Empty{}
\newcommand\oplabel[1]{
  \def\OpArg{#1} \ifx \OpArg\Empty {} \else
  	\label{#1}
  \fi}
		